\documentclass[11pt]{amsart}

\usepackage{pgf,pgfarrows,pgfnodes,pgfautomata,pgfheaps,pgfshade,hyperref, amssymb,enumerate,amsmath}
\usepackage{amsfonts}
\usepackage{geometry}
\usepackage{graphics}
\usepackage{color}

\usepackage[all]{xy}

\usepackage{hyperref}
\usepackage{bbm}
\usepackage{tikz}
\usepackage{ulem}
\usetikzlibrary{positioning, arrows.meta}

\usepackage[capitalize]{cleveref}
\usepackage{mathtools}
\usepackage[shortlabels]{enumitem}
\usepackage{stmaryrd}
\usepackage{tcolorbox}

\newcommand{\NN}{{\mathbb N}}
\newcommand{\ZZ}{\mathbb Z}
\newcommand{\QQ}{{\mathbb Q}}

\newcommand{\FF}{{\mathbb F}}

\def\M{{\mathcal M}}

\def\P{{\mathcal P}}

\def\Per{{\rm Per}}
\def\Aper{{\rm Aper}}
\def\Stab{{\rm Stab}}
\numberwithin{equation}{section}

\newtheorem{theo}{Theorem}
\newtheorem{proposition}[theo]{Proposition}

\newtheorem{coro}[theo]{Corollary}
\newtheorem{lemma}[theo]{Lemma}
\newtheorem{defi}[theo]{Definition}
\newtheorem{remark}[theo]{Remark}

\RequirePackage{color}\definecolor{RED}{rgb}{1,0,0}\definecolor{BLUE}{rgb}{0,0,1} 
\author[]{Mar\'{\i}a Isabel Cortez}
\address{Facultad de Matem\'aticas, Pontificia Universidad Cat\'olica de Chile. Edificio Rolando Chuaqui, Campus San Joaquín. Avda. Vicuña Mackenna 4860, Macul, Chile.}
\email{maria.cortez@uc.cl}

\thanks{Both authors thanks Lorentz Center, where part of this work was done.}

\author[]{Jaime G\'omez}
\address{Facultad de Matem\'aticas, Pontificia Universidad Cat\'olica de Chile. Edificio Rolando Chuaqui, Campus San Joaquín. Avda. Vicuña Mackenna 4860, Macul, Chile.}
\email{jagomez7@uc.cl}

\thanks{J. Gómez was supported by ANID/ Doctorado Nacional No. 21200054}

\begin{document}

\title{Almost 1-1 extensions of Furstenberg-Weiss type and  test for amenability.}

\begin{abstract}
Let $G$ be an infinite residually finite group. We show that for every minimal equicontinuous Cantor system $(Z,G)$ with a free orbit, and for every minimal extension $(Y,G)$ of $(Z,G)$,   there exist a minimal almost 1-1 extension $(X,G)$ of $(Z,G)$ and a Borel equivariant map $\psi:Y\to X$  that induces an affine bijection $\psi^*$ between $\M(Y,G)$ and $\M(X,G)$,  the spaces of invariant probability measures of $(Y,G)$ and $(X,G)$, respectively. If $Y$ is a Cantor set, then $(Y,G)$ and $(X,G)$    are  Borel isomorphic, i.e., $\psi^*$ is also a homeomorphism.  As an application, we show that the family of Toeplitz subshifts is a test for amenability for residually finite groups, i.e., a residually finite group $G$ is amenable if and only if every Toeplitz $G$-subshift has invariant probability measures.
 \end{abstract}

\maketitle

\section{Introduction}

Let $G$ be a countable infinite group. In \cite{FW89} (see also \cite{W12}), Furstenberg and Weiss prove 
that if there is a topological factor map $\pi:Y\to Z$ from a transitive system  $(Y,G)$ to a minimal system $(Z,G)$, with $Z$ being totally disconnected, then there exist a minimal dynamical system $(X,G)$, an almost one-to-one factor map $\tau:X\to Z$, and a Borel one-to-one equivariant map $\theta:Y_0\subseteq Y\to X$, where $Y_0$ is the pre-image by $\pi$ of a  full measure set $Z_0$ of $Z$, such that $\pi=\tau\circ\theta$. As a consequence, the invariant probability measures of $(X,G)$  supported on $\theta(Y_0)$ are in a bijection with the invariant probability measures of $(Y,G)$. In other words, the almost 1-1 extension $(X,G)$ has at least as many invariant probability measures as $(Y,G)$.  Later, in  \cite{DL98}, Downarowicz and Lacroix improve this result for the case $G=\ZZ$, relaxing some hypotheses on $(Y,\ZZ)$ and $(Z,\ZZ)$, and building a Borel map $\theta$  that induces an affine homeomorphism between the spaces of invariant probability measures of $(Y,\ZZ)$ and $(X,\ZZ)$. In this case, $(Y,\ZZ)$ and the almost 1-1 extension $(X,\ZZ)$ have the same spaces of invariant probability measures. 

This article focuses on the case that $G$ is an infinite residually finite group, $(Z,G)$ is a minimal equicontinuous Cantor system with a free orbit, and $(Y,G)$ is a minimal extension of $(Z,G)$. Inspired by the ideas in \cite{FW89}, we show it is possible to construct $(X,G)$ in order that the map $\theta$ induces an   affine bijection (a homeomorphism if $Y$ is totally disconnected) between 
$\M(Y,G)$ and $\M(X,G)$, the spaces of invariant probability measures of $(Y,G)$ and $(X,G)$, respectively.  In particular, if $G$ is non-amenable and $(Y,G)$ has no invariant measures, then neither does $(X,G)$. Thus, as a corollary, we get that if $G$ is non-amenable, then every minimal equicontinuous action of $G$ on the Cantor set with a free orbit has a minimal almost 1-1 extension with no invariant measures. This result  answers the question posed by Veech in  \cite[Page 823]{V77} about the existence of almost automorphic systems without invariant measures, which was previously  announced by Furstenberg and Weiss in \cite{FW89}, and recently shown  by Tsankov \cite{T24}, using different arguments.


Our first result is stated in the following theorem.

\begin{theo}\label{theorem}  Let $G$ be an infinite residually finite group. Let $(Z,G)$ be a minimal equicontinuous system having a free orbit and such that $Z$ is a Cantor set. Let $(Y,G)$ be a minimal   extension of $(Z,G)$ through a topological factor map $\pi:Y\to Z$. Then there exist a minimal    extension $(X,G)$ of $(Z,G)$ through an almost 1-1 factor map $\tau:X\to Z$,  and a Borel equivariant map $\psi:Y\to X$, satisfying the following:

\begin{enumerate}
\item  $\tau\circ\psi=\pi.$
\item $\psi$ is one-to-one on a set $\pi^{-1}(\tilde{Z})$, where $\tilde{Z}\subseteq Z$ is a full-measure set with respect to the unique invariant probability measure in $(Z,G)$.
\item  The map $\psi$ induces  an affine  bijection 
between $\M(Y,G)$ and $\M(X,G)$, given by $\psi^*(\mu)(A)=\mu(\psi^{-1}(A))$, for every $\mu\in \M(Y,G)$ and every Borel set $A\subseteq X$.
\item If in addition, $Y$ is a Cantor set, then $X$ is also a Cantor set and the map $\psi^*$ is an affine homeomorphism between $\M(Y,G)$ and $\M(X,G)$. 
\end{enumerate}  
\end{theo}

 In the context of Theorem \ref{theorem}, points (1) and (2)  do not  guarantee that $\psi^*$ is a bijection between $\M(X,G)$ and $\M(Y,G)$. Indeed, suppose that $(X,G)$ is a symbolic almost 1-1 extension of $(Z,G)$ having two different ergodic measures $\mu_1$ and $\mu_2$ with the property that the almost 1-1 factor map $\tau: X\to Z$ is a measure-theoretic conjugacy between $(X,G,\mu_i)$ and $(Z,G,\nu)$, for $i=1,2$ (such examples exist, see for example, \cite{CCG23} or \cite{Wi84}).
 Taking  $Y=Z$, $\pi=id$ and $\psi=(\tau|_{S_1})^{-1}$, with $S_1\subseteq X$ a $\mu_1$-full measure set such that the restriction of $\tau$ to $S_1$ makes  a measure-theoretic conjugacy between $(X,G,\mu_i)$ and $(Z,G,\nu)$, we get   conditions (1) and (2) of Theorem \ref{theorem}. Nevertheless,  $\psi^*$ is not surjective because  $\mu_2$ (and any other invariant measure $\mu\neq \mu_1$) has not pre-image by $\psi^*$.    Section \ref{Fibers} is devoted to show that in our construction, $\psi^*$ is a bijection between the spaces of invariant probability measures, and a homeomorphism when the phase spaces are Cantor sets. It is important to note that Theorem \ref{theorem} is valid for a more restricted class of systems than those in \cite{FW89} and \cite{W12}. 

\bigskip

  As an application of Theorem \ref{theorem} we get the following result.

\begin{theo}\label{Main-Toeplitz-measures}
Let $G$ be a non-amenable residually finite group. Then for every  minimal equicontinuous Cantor system $(Z,G)$ having a free orbit,   there  exists a Toeplitz $G$-subshift  with no invariant probability measures, whose maximal equicontinuous factor is  $(Z,G)$.
\end{theo}

Theorem \ref{Main-Toeplitz-measures} shows that the class of Toeplitz subshifts is a test for amenability for residually finite groups. See  \cite{FJ21,GH97,Sew14}  for more examples of families of dynamical systems which are test for amenability for countable groups. 

If $G$ is amenable and residually finite, then every metrizable Choquet simplex is realized as the space of invariant probability measures of a Toeplitz $G$-subshift (see \cite{CP14, Do91}). Theorem \ref{Main-Toeplitz-measures} can be interpreted as saying that when $G$ is a non-amenable residually finite group, then the empty simplex is realized among the family of Toeplitz $G$-subshifts.

As a corollary of Theorem \ref{Main-Toeplitz-measures}, we get the following characterization of amenability of re\-si\-dua\-lly finite groups:

\begin{coro}\label{corollary}
Let $G$ be an infinite residually finite countable  group. The following statements are equivalent:
\begin{enumerate}
\item $G$ is amenable.
\item Every almost automorphic topological dynamical system $(X,G)$ has invariant pro\-ba\-bi\-li\-ty measures.
\item Every almost automorphic topological dynamical system $(X,G)$, whose maximal equicontinuous factor is a Cantor system, has invariant probability measures.
\item Every Toeplitz $G$-subshift has invariant probability measures.

\item There exists a minimal equicontinuous Cantor system  $(Z,G)$ with a free orbit, such that every symbolic almost one-to-one extension of $(Z,G)$ has invariant probability measures.
\end{enumerate}

\end{coro}


 This article is organized as follows: in Section \ref{definitions} we recall basic concepts as odometers, Toeplitz subshifts, equicontinuous systems and residually finite groups. In Section \ref{Good-subsequences} we show that if $(Z,G)$ is a minimal equicontinuous Cantor system with a free orbit, conjugate to the odometer associated to a decreasing sequence $(\Gamma_n)_{n\in\NN}$ of finite index subgroups of $G$ with trivial intersection, then there exists an increasing sequence $(D_n)_{n\in\NN}$ of finite subsets of $G$ satisfying several properties, including that every $D_n$ is a set of right coset representatives of $G/\Gamma_n$.  
 In Section \ref{Section-borel-map}  we prove Theorem \ref{theorem}, and in Section \ref{General-case}, we prove Theorem \ref{Main-Toeplitz-measures} and Corollary \ref{corollary}.

\subsection*{Acknowledgements} The authors warmly thank Sebasti\'an Barbieri, Raimundo Brice\~no, Paulina Cecchi, Sebasti\'an Donoso, Tomasz Downarowicz, Matthieu Joseph, Patricio Pérez, Samuel Petite,  and Todor Tsankov for fruitful discussions.

\section{Definitions and background.}\label{definitions}

A {\it topological dynamical system} is a pair $(X,G)$ where $X$ is a compact metric space and $G$ is a countable discrete group acting continuously on $X$. We call $X$ the {\it phase space} of the system. For every $x\in X$ and $g\in G$, we denote $gx$ the action of $g$ on $x$. The {\it stabilizer} of $x\in X$ is the subgroup $\Stab(x)=\{g\in G: gx=x\}$. The dynamical system $(X,G)$ is  {\it free} if $\Stab(x)=\{1_G\}$, for every $x\in X$. We say that $(X,G)$ {\it has a free orbit}, if there exists $x\in X$ such that $\Stab(x)=\{1_G\}$. A subset $A\subseteq X$ is $G$-{\it invariant} or just {\it invariant}, if $gA=A$ for every $g\in G$, where $gA=\{gx: x\in A\}$.  The topological dynamical system $(X,G)$ is {\it minimal} if for every $x\in X$ its orbit $O_G(x)=\{gx: g\in G\}$ is dense in $X$.  A {\it minimal component} of $(X,G)$ is a closed non-empty $G$-invariant subset $Y$ of $X$ such that $(Y,G)$ is minimal. It is well known that every topological dynamical system has minimal components (see for example \cite{Au88}). The topological dynamical system $(X,G)$ is {\it equicontinuous} if  for every $\varepsilon>0$ there exists $\delta>0$, such that if $d(x,y)<\delta$ then $d(gx,gy)<\varepsilon$, for every $x,y\in X$ and $g\in G$.

Let $(X,G)$ and $(Y,G)$ be two topological dynamical systems. We say that a map $\pi:X\to Y$ is {\it equivariant} if $\pi(gx)=g\pi(x)$ for every $g\in G$ and $x\in X$. The equivariant map $\pi$ is a {\it factor map} if it is also continuous and  surjective. If $\pi:X\to Y$ is a factor map, we say that $(X,G)$ is a {\it topological extension} or an extension of $(Y,G)$, and $(Y,G)$ is a {\it topological factor} or a factor of $(X,G)$. If $\pi$ is bijective, we say that $(X,G)$ and $(Y,G)$ are {\it conjugate}. The factor map $\pi:X\to Y$ is {\it almost one-to-one} if the set $\{y \in Y: |\pi^{-1}\{\pi(y)\}|=1\}$ is dense in $Y$.  

Every topological dynamical system $(X,G)$ has a unique (up to conjugacy) {\it maximal equicontinuous factor}, that is, a factor $(Y,G)$ which is equicontinuous, and such that every other factor from $(X,G)$ to an equicontinuous system, passes through the factor from $(X,G)$ to $(Y,G)$. A system $(X,G)$ is {\it almost automorphic} if its maximal equicontinuous factor is minimal, and if the associated factor map is almost one-to-one.

\begin{lemma}\label{minimality}
Let $(Y,G)$ be a minimal topological dynamical system. Let $y_0\in Y$.  For every $\varepsilon>0$ there exists a finite set $F\subseteq G$, such that for every $y\in Y$ and for every $g\in G$, there exists $h\in gF$ such that $d(h^{-1}y_0, y)<\varepsilon$.
\end{lemma}

\begin{proof}
Suppose this is not true. Then there exists $\varepsilon>0$, such that for every finite set $F\subseteq G$, there exist $y\in Y$ and $g\in G$, such that for every $h\in Fg$ we have $d(h^{-1}y_0,y)\geq \varepsilon$.  Let $(F_n)_{n\in \NN}$ be an increasing sequence of finite subsets of $G$ such that $1_G\in F_1$ and  whose union is equal to $G$. Let $g_n\in G$ and $y_n\in Y$ be such that for every $h\in g_nF_n$ we have $d(h^{-1}y_0,y_n)\geq \varepsilon$.

After taking subsequences if necessary, we can assume that $(g_n^{-1}y_0)_{n\in\NN}$ and $(y_n)_{n\in\NN}$ converge to $x\in Y$ and $y\in Y$ respectively. Let $h\in G$. Let $n_0\in \NN$ be such that $h\in F_n$, for every $n\geq n_0$. We have $g_nh\in g_nF_n$, for every $n\geq n_0$. Then
$$
d(h^{-1}g_n^{-1}y_0, y_n)\geq \varepsilon, \mbox{ for every } n\geq n_0.
$$
From this we get
$$
\varepsilon \leq d(h^{-1}g_n^{-1}y_0, y_n)\leq d(h^{-1}g_n^{-1}y_0, h^{-1}x)+d(h^{-1}x,y)+d(y, y_n), 
$$
 for every  $n\geq n_0$, which implies that
$$
\varepsilon\leq d(h^{-1}x, y) \mbox{ for every } h\in G,
$$
which is not possible, because $(Y,G)$ is minimal.
\end{proof}

\subsection{Invariant measures} Let $(X,G)$ be a topological dynamical system. We denote by $\M(X)$ the space of Borel probability measures of $X$. This space is convex, compact, and metrizable   with respect to the weak topology (see \cite{DGS}, for details). An {\it invariant probability measure } of   $(X,G)$ is a measure $\mu\in \M(X)$ such that $\mu(gA)=\mu(A)$, for every $g\in G$ and every Borel set $A\subseteq X$. We denote $\M(X,G)$ the space of invariant probability measures of $(X,G)$.  

It is well known that the group $G$ is {\it amenable} if and only if for every topological dynamical system $(X,G)$ the space $\M(X,G)$ is non-empty (see, for example,  \cite{KL16}). 

 The next result is classical in ergodic theory  (see, for example,  \cite{DGS}). 
 
 \begin{lemma}\label{sobre} 
Let  $(X,G)$ and $(Y, G)$ be  topological dynamical systems.    If  $\pi: X\to Y$  is a factor map, then the function $\pi^*:\M(X)\to \M(Y)$ defined as
$$
\pi^*\mu(A)=\mu(\pi^{-1}(A)), \mbox{ for every  } \mu\in\M(X) \mbox{ and every Borel set }  A\subseteq Y, 
$$ 
is affine, continuous and surjective.  
\end{lemma}
In the context of Lemma \ref{sobre}, if the factor $(Y,G)$ has no invariant probability measures, then neither does the extension $(X,G)$. In other words, extensions of topological dynamical systems with no invariant measures, have no invariant measures.

\subsection{Cantor systems.}

We say that the topological dynamical system $(X,G)$ is a {\it Cantor system} if $X$ is a Cantor set, i.e., non-empty, compact, metric, totally disconnected and without isolated points. Examples of  Cantor systems are given by the  $G$-subshifts.

\subsubsection{$G$-subshifts.}  Let $\Sigma$ be a finite alphabet, i.e., a finite set with at least two elements. Consider the set $$\Sigma^G=\{x=(x(g))_{g\in G}: x(g)\in \Sigma, \mbox{ for every } g\in G\}.$$
The space $\Sigma^G$ is a Cantor set when we endow $\Sigma$ with the discrete topology and $\Sigma^G$ with the product topology.  
The group $G$ acts continuously on $\Sigma^G$ by means of the {\it shift action}, which is defined as follows: let $x=(x(g))_{g\in G}\in \Sigma^G$ and $h\in G$,
$$
hx(g)=x(h^{-1}g), \mbox{ for every } g\in G.
$$

The Cantor system $(\Sigma^G,G)$ is called the $|\Sigma|$-full $G$-shift, where $|\Sigma|$ denotes the cardinality of $\Sigma$. A {\it subshift} of $\Sigma^G$ is a closed $G$-invariant subset $X\subseteq \Sigma^G$. We  also call subshift the topological dynamical system $(X,G)$ given by the restriction of the shift-action of $G$ on $X$ (see \cite{CC10} for more details on subshifts).  The next lemma shows that minimal infinite subshifts are Cantor systems.

\begin{lemma}\label{Cantor}
Let $(X,G)$ be a minimal topological dynamical system such that $X$ is totally disconnected. If $X$ is infinite, then $X$ has no isolated points.
\end{lemma}
\begin{proof}
Suppose that $x\in X$ is an isolated point. Then $\{x\}$ is an open set of $X$. Since the system $(X,G)$ is minimal, the set of return times of $x$ to $\{x\}$  is syndetic (see \cite{Au88}), which is equivalent to say that $\{g\in G: gx=x\}$ is a finite index subgroup of $G$. This implies that the orbit of $x$ is finite. Thus if $X$ is infinite, then $X$ has no isolated points.
\end{proof}

\subsubsection{Metric on the space of measures of a Cantor set.}
 
Let $X$ be a Cantor space and let 
$\mathcal{K}=(\mathcal{K}_n)_{n\in\NN}$ be a decreasing sequence of nested clopen partitions of $X$ that spans its topology. Observe that every partition $\mathcal{K}_n$ is finite, and the diameter of the elements in $\mathcal{K}_n$ goes to zero with $n$.
Let define  
$$
\mathcal{S}(\mathcal{K})=\left\{ \sum_{K\in \mathcal{K}_n} a_K\mathbbm{1}_{K}: a_K\in\QQ, \mbox{ for every } K\in \mathcal{K}_n, n\in\NN\right\}.
$$

\begin{lemma}\label{Lemma21}
    For every $\varepsilon>0$ and $f\in C(X)$, there exists $f_\varepsilon\in \mathcal{S}(\mathcal{K})$ such that\\ $\sup_{x\in X}|f(x)-f_\varepsilon(x)|<\varepsilon$.
\end{lemma}
\begin{proof}
    Let $\varepsilon>0$, $f\in C(X)$ and $\delta>0$ be such that if $d(x,y)<\delta$ then $|f(x)-f(y)|<\frac{\varepsilon}{2}$. Let $n\in\NN$ be such that the elements of $\mathcal{K}_n$ have diameter smaller than $\delta$.
   For every $K\in \mathcal{K}_n$, we choose $x_K\in K$ and $a_K\in\QQ$ such that $|a_K-f(x_K)|<\frac{\varepsilon}{2}$.
   We define $f_\varepsilon=\sum_{K\in\mathcal{K}_n} a_K\mathbbm{1}_{K}$. 
   Since for every $x\in K$ we have $d(x_K,x)<\delta$, we get $|f(x)-a_K|\leq |f(x)-f(x_K)|+|f(x_K)-a_K|<\varepsilon$. 
   This implies $\sup_{x\in X}|f(x)-f_\varepsilon(x)|<\varepsilon$.
\end{proof}
For every $\mu_1$ and $\mu_2$ in $\M(X)$, we define
\begin{align}\label{newmetric}
d(\mu_1,\mu_2)=\sum_{n=1}^{\infty}\sum_{K\in \mathcal{K}_n}\frac{1}{2^ns_n}|\mu_1(K)-\mu_2(K)|, 
\end{align}
where $s_n$ is the number of elements in $\mathcal{K}_n$, for every $n\in\NN$.  It is straightforward to check that $d$ is a metric on $\M(X)$. 
\begin{lemma}\label{Lemma22}
Let $g\in\mathcal{S}(\mathcal{K})$.
For every $\varepsilon>0$, there exists $\delta=\delta(g,\varepsilon)>0$ such that   if $\mu_1,\mu_2\in\mathcal{M}(X)$  verify $d(\mu_1,\mu_2)<\delta$, then  $\left|\int gd\mu_1-\int gd\mu_2\right|<\varepsilon$.
\end{lemma}
\begin{proof}
Let $g\in \mathcal{S}(\mathcal{K})$ and $n\in\NN$ be such that $g=\sum_{K\in \mathcal{K}_n} a_K\mathbbm{1}_{K}$, for some $a_K\in \QQ$, for every $K\in \mathcal{K}_n$.
Let $\delta:=\min\{\frac{\varepsilon}{2^{n}|a_K|s_n^2}\colon K\in \mathcal{K}_n\}$. Thus,  if $\mu_1$ and $\mu_2$ are two measures verifying $d(\mu_1,\mu_2)<\delta$, then
$$\frac{1}{2^{n}s_n}|\mu_1(K)-\mu_2(K)|<d(\mu_1,\mu_2)<\frac{\varepsilon}{2^{n}|a_K|s_n^2}\mbox{ for every } K\in \mathcal{K}_n.$$
This implies
\begin{align*}
        \left|\int gd\mu_1-\int gd\mu_2\right|\leq\sum_{K\in \mathcal{K}_n}|a_K||\mu_1(K)-\mu_2(K)|\leq\sum_{K\in \mathcal{K}_n}|a_K|\frac{\varepsilon}{|a_K|s_n}\leq \varepsilon.
    \end{align*}  
\end{proof}

\begin{coro}\label{Coro23}
Let $f\in C(X)$.
For every $\varepsilon>0$, there exists $\delta=\delta(f,\varepsilon)>0$ such that if  $\mu_1,\mu_2\in\mathcal{M}(X)$  satisfy $d(\mu_1,\mu_2)<\delta$, then 
 $\left|\int fd\mu_1-\int fd\mu_2\right|<\varepsilon$.
\end{coro}
\begin{proof}
Let $f\in C(X)$ and $\varepsilon>0$.
By Lemma \ref{Lemma21}, there exists $g\in\mathcal{S}(\mathcal{K})$ such that $\sup_{x\in X}|f(x)-g(x)|<\frac{\varepsilon}{3}$.
By Lemma \ref{Lemma22}, there exists $\delta'=\delta'(g,\varepsilon/3)>0$ such that if $\mu_1,\mu_2\in\mathcal{M}(X)$ verify $d(\mu_1,\mu_2)<\delta'$, then   $\left|\int gd\mu_1-\int gd\mu_2\right|<\frac{\varepsilon}{3}$.
Taking $\delta:=\delta'$ we conclude.
\end{proof}
  
\begin{coro}\label{Coro24}
Let $(\mu_i)_{i\in\NN}$ be a sequence of measures in $\M(X)$ and $\mu\in \M(X)$. The following statements are equivalent:
\begin{enumerate}
\item The sequence $(\mu_i)_{i\in\NN}$ converges weakly to $\mu$;
\item $\lim_{i\to\infty}d(\mu_i,\mu)=0$.
\end{enumerate}
\end{coro}

\begin{proof} 
Suppose that $(\mu_i)_{i\in\NN}$ converges weakly to $\mu$. Then for every clopen set $C\subseteq X$, we have $\lim_{i\to\infty}\mu_i(C)=\mu(C)$. In particular, for every $n\in \NN$ and $K\in \mathcal{K}_n$, we have $\lim_{i\to\infty}\mu_i(K)=\mu(K)$. Let $\varepsilon>0$  and $m\in\NN$ be such that $\sum_{n>m}\frac{1}{2^n}<\frac{\varepsilon}{2}.$ Let $i_0\in\NN$ be such that 
$$
|\mu_i(K)-\mu(K)|<\frac{\varepsilon}{2}\left(\sum_{n=1}^m\frac{1}{2^n} \right)^{-1}, \mbox{ for every  } K\in \mathcal{K}_j, \mbox{ with } 1\leq j\leq m.
$$
Then, for every $i\geq i_0$,
\begin{eqnarray*}
d(\mu_i,\mu) &   = & \sum_{n=1}^{\infty}\sum_{K\in \mathcal{K}_n}\frac{1}{2^ns_n}|\mu_i(K)-\mu(K)|\\
             & =& \sum_{n=1}^{m}\sum_{K\in \mathcal{K}_n}\frac{1}{2^ns_n}|\mu_i(K)-\mu(K)|+ \sum_{n>m}\sum_{K\in \mathcal{K}_n}\frac{1}{2^ns_n}|\mu_i(K)-\mu(K)|\\
             &\leq & \frac{\varepsilon}{2} +\frac{\varepsilon}{2}\\
             &=& \varepsilon.
\end{eqnarray*}
From this we get (1) implies (2).

Conversely, let $f\in C(X)$ and $\varepsilon>0$. Let $\delta>0$ be as in Corollary \ref{Coro23}, i.e., such that if $d(\nu_1,\nu_2)<\delta$ then $\left|\int fd\nu_1-\int fd\nu_2\right|<\varepsilon$. Since $\lim_{i\to\infty}d(\mu_i, \mu)=0$, there exists $i_0\in\NN$ such that $d(\mu_i,\mu)<\delta$, for every $i\geq i_0$. This implies that 
$\left|\int fd\mu_i-\int fd\mu\right|<\varepsilon$, for every $i\geq i_0$. Since $\varepsilon>0$  and $f\in C(X)$ are arbitrary, we get
$$
\lim_{i\to\infty}\int f d\mu_i=\int f d\mu, \mbox{ for every } f\in C(X).
$$
This shows $(2)$ implies (1).
\end{proof}

\subsection{Residually finite groups and their Cantor systems.} \label{residually-finite}

 The group $G$ is {\it residually finite} if for every $g\in G\setminus\{1_G\}$ there exist a finite group $F$ and a group homomorphism $\phi:G\to F$ such that $\phi(g)\neq 1_F$. This property is equivalent to the existence of a decreasing sequence $(\Gamma_n)_{n\in\NN}$ of finite index subgroups of $G$ such that $\bigcap_{n\in\NN}\Gamma_n=\{1_G\}$. 
 Examples of residually finite groups are the free group $\FF_n$ with $n$ generators and $\ZZ^n$, for every $n\in\NN$ (see \cite{CC10} for more details about residually finite groups).

\subsubsection{Odometers}\label{odometers} 
Let $(\Gamma_n)_{n\in\NN}$ be a strictly decreasing sequence of finite index subgroups of $G$. The \textit{$G$-odometer} associated to $(\Gamma_n)_{n\in\mathbb{N}}$ is defined as
 \begin{eqnarray*}
    Z & = & \varprojlim(G/\Gamma_n,\tau_n)\\
    &= &\{(x_n)_{n\in\mathbb{N}}\in\prod_{n\in\mathbb{N}}G/\Gamma_n: \tau_n(x_{n+1})=x_n,\mbox{ for every }n\in\mathbb{N}\},
 \end{eqnarray*}
where $\tau_n\colon G/\Gamma_{n+1}\to G/\Gamma_n$ is the canonical projection, for every $n\in\mathbb{N}$. The space $Z$ is a Cantor set when we endow every $G/\Gamma_n$ with the discrete topology,  $\prod_{n\in\mathbb{N}}G/\Gamma_n$ with the product topology and  $Z$ with the induced topology.  Let $\tau:G\to Z$ be given by $\tau(g)=(g\Gamma_n)_{n\in\NN}$, for every $g\in G$. The group $G$ acts on $Z$ as $g(g_n\Gamma_n)_{n\in\NN}=(gg_n\Gamma_n)_{n\in\NN}$, for every $g\in G$ and $(g_n\Gamma_n)_{n\in\NN} \in Z$. The topological dynamical system $(Z,G)$ is
known as the $G$-{\it odometer} or odometer associated to the sequence $(\Gamma_n)_{n\in \NN}$. This system is 
 equicontinuous and minimal.  The $G$-odometer is free if and only if $\bigcap_{n\in\NN}g_n\Gamma_ng_n^{-1}=\{1_G\}$, for every $(g_n\Gamma_n)_{n\in\NN}\in Z$, which implies that there exists a free $G$-odometer if and only if $G$ is residually finite (see \cite{CP08} for more details about odometers).

\medskip

 Let $(Z,G)$ be the $G$-odometer associated to the decreasing sequence $(\Gamma_n)_{n\in\NN}$ of finite index subgroups of $G$.

For every $n\in\NN$, we define  $$C_n=\{(g_i\Gamma_i)_{i\in\NN}\in Z: g_n\Gamma_n=1_G\Gamma_n\},$$
and
$$
\P_n=\{gC_n: g\in D^{-1}_n, n\in\NN\},
$$
where $D_n\subseteq G$ is a set containing exactly one element of each class in $\{\Gamma_ng: g\in G\}$.
It is straightforward to check that the topology of $Z$ is generated by  $(\P_n)_{n\in\NN}$. Furthermore, $(\P_n)_{n\in\NN}$ is a sequence of nested clopen partitions, and for every $n\in\NN$, the intersection $gC_n\cap hC_n$ is non-empty if and only $g\in h\Gamma_n$, which is equivalent to $gC_n=hC_n$ (see, for example, \cite{CP08}). We call the sets $gC_n$, {\it basic clopen sets}. Let $\nu$ be the unique invariant probability measure of $(Z,G)$. Observe that
$$
\nu(gC_n)=\frac{1}{|D_n|}=\frac{1}{[G:\Gamma_n]}, \mbox{ for every } g\in G \mbox{ and } n\in\NN.
$$

\begin{lemma}\label{free-orbit}
 A $G$-odometer has a free orbit if and only if it is conjugate  to a $G$-odometer associated to a decreasing sequence of finite index subgroups of $G$ with trivial intersection.  
\end{lemma}
\begin{proof}
Let $Z$ be a $G$-odometer associated to the sequence $(\Gamma_n)_{n\in\NN}$  of finite index subgroups of $G$. Let $z=(z_n\Gamma_n)_{n\in\NN}\in Z$  be such that   $\Stab(z)=\bigcap_{n\in\NN}z_n\Gamma_nz_n^{-1}$  is trivial.   For every $n\in\NN$, let define $\tilde{\Gamma}_n=z_n\Gamma_nz_n^{-1}$. Let $Y$ be the $G$-odometer associated to the sequence $(\tilde{\Gamma}_n)_{n\in\NN}$. It is not difficult to show that the function $f: Y\to Z$, given by $f((g_n\tilde{\Gamma}_n)_{n\in\NN})=(g_nz_n\Gamma_n)_{n\in\NN}$ for every $(g_n\tilde{\Gamma}_n)_{n\in\NN}\in Y$, is a topological conjugacy.
\end{proof}


\medskip

 \begin{proposition}\label{equicontinuous-characterization}
  $(X,G)$  is a minimal equicontinuous Cantor system if and only if $(X,G)$ is conjugate to a $G$-odometer.     
 \end{proposition}
\begin{proof}
  It is sufficient to follow the proof of \cite[Theorem 2.7]{CM16}, and to note that if $(X,G)$ is not free, then the function $\varphi:G\to K$ is not necessarily injective. 
  This implies that the action of $G$ on $K$ is not necessarily free.   The intersection of the groups $G_n$'s mentioned there, is the stabilizer of the neutral element $e\in K$ (and maybe non-trivial). Then, the proof follows  to conclude that $(X,G)$ is regularly recurrent, meaning that for every $x\in X$ and  every clopen neighborhood $C$ of $x$, there exists a finite index subgroup $\Gamma$ of $G$ such that $\gamma x\in C$, for every $\gamma\in \Gamma$.  Let $x\in X$ and $(U_n)_{n\in\NN}$ be a nested sequence of clopen sets such that $\bigcap_{n\in\NN} U_n=\{x\}$. Then, by \cite[Lemma 4]{CP08}, for every $n\in\NN$ there exists a finite index subgroup $\Gamma_n$ such that $C_n=\overline{O_{\Gamma_n}(x)}\subseteq U_n$ is a clopen set verifying $C_n\cap \gamma C_n\neq \emptyset$ if and only if $\gamma C_n=C_n$, and  if and only if $\gamma\in \Gamma_n$. The system $(X,G)$ is conjugate to the odometer given by the sequence $(\Gamma_n)_{n\in\NN}$ (see the proof of \cite[Theorem 2]{CP08}).
\end{proof}

The following result is a direct consequence of Lemma \ref{free-orbit} and Proposition \ref{equicontinuous-characterization}.

\begin{proposition}\label{group-characterization}
Let $G$ be an infinite countable group. The following statements are equivalent:
\begin{enumerate}
\item $G$ is residually finite
\item There exists a free  $G$-odometer.
\item There exists a free minimal equicontinuous Cantor system $(X,G)$.
\item There exists a  minimal equicontinuous Cantor system $(X,G)$ with a free orbit.
\end{enumerate}
\end{proposition}

\subsubsection{Toeplitz subshifts} Let $\Sigma$ be a finite alphabet. An element $x\in \Sigma^G$ is {\it Toeplitz} if for every $g\in G$ there exists a finite index subgroup $\Gamma$ of $G$ such that $x(g)=x(\gamma^{-1} g)$, for every $\gamma\in \Gamma$. We say that $\Gamma$ is a {\it period } of $x$ or that the coordinate $g$ is periodic with respect to $\Gamma$.  A subshift $X\subseteq \Sigma^G$ is a {\it Toeplitz subshift} if there exists a  Toeplitz element $x\in \Sigma^G$,  such that $X=\overline{O_G(x)}$. The Toeplitz subshifts are always  minimal (see \cite{CP08, Dow05, Kr07, Kr10} for more details on Toeplitz subshifts).
 Observe that if $x\in \Sigma^G$ is a non-periodic Toeplitz subshift (i.e., its stabilizer  with respect to the shift action is trivial), then the intersection of its periods is trivial. From this observation we deduce that the existence of non-periodic Toeplitz elements in $\Sigma^{G}$  implies that $G$ is residually finite. The following proposition tell us the  converse holds as well.
 

 

\begin{proposition}\label{extension-Toeplitz}(\cite[Theorem 7]{CP08}) Let $(\Gamma_n)_{n\in\NN}$ be a strictly decreasing sequence of finite index  subgroups (not necessarily normal or with trivial intersection) of $G$.
Then there exists a Toeplitz subshift $X\subseteq \{0,1\}^G$ such that $(X,G)$ is an almost 1-1 extension of the $G$-odometer associated to $(\Gamma_n)_{n\in\NN}$. 
\end{proposition}

\subsubsection{Period structure of a Toeplitz element.}

Let $x\in \Sigma^G$ be a Toeplitz element. Suppose that the finite index subgroup $\Gamma$ of $G$ is a period of $x$.  We define  

\begin{align*}
    \Per(x,\Gamma,\alpha)&=\{g\in G\colon x(\gamma g)=\alpha \mbox{ for every }\gamma\in\Gamma\}, \mbox{ for every } \alpha\in \Sigma.\\
    \Per(x,\Gamma)&=\bigcup_{\alpha\in \Sigma}\Per(x,\Gamma,\alpha).
\end{align*}
The elements of $\Per(x,\Gamma)$ are those belonging to some coset $\Gamma g $ for which $x$ restricted to $\Gamma g$ is constant.



For every $g\in G$, for every $\alpha\in\Sigma$  and every   period $\Gamma$ of $x$, we have   $g\Per(x,g^{-1}\Gamma g,\alpha)=\Per(g x,\Gamma,\alpha)$. The  period $\Gamma$ of $x$ is an {\it essential period of} $x$ if $\Per(x,\Gamma,\alpha)\subseteq  \Per(g x,\Gamma,\alpha)$ for every $\alpha\in \Sigma$, implies   $g\in \Gamma$.

 \begin{lemma}(\cite[Lemma 2.4]{CCG23}, \cite{CP08})
Let $x\in\Sigma^G$ be a Toeplitz element. For every period $\Gamma$ of $x$, there exists an essential  period $K$ of $x$ such that $\Per(x,\Gamma,\alpha)\subseteq\Per(x, K, \alpha)$, for every $\alpha\in\Sigma$.
 \end{lemma}
 
 A {\it period structure} of $x$ is a nested sequence  $(\Gamma_n)_{n\in\mathbb{N}}$ of   essential periods of $x$, such that $G=\bigcup_{n\in\mathbb{N}}\Gamma_n$. A period structure of $x$ always exists (see \cite[Corollary 6]{CP08}).
  Let $(\Gamma_n)_{n\in\mathbb{N}}$ be a period structure of $x$. Let $X=\overline{O_{G}(x)}$ be the associated Toeplitz subshift. For each $n\in\mathbb{N}$, consider the set
 $$B_n=\{ y\in X: \Per(y,\Gamma_n, \alpha)=\Per(x, \Gamma_n, \alpha), \mbox{ for every } \alpha\in \Sigma\}.$$
It is straightforward to verify that $gx\in B_n$ if and only if $g\in \Gamma_n$. This implies the following result.
 \begin{lemma}(\cite[Lemma 8 and Proposition 6]{CP08} or \cite[Lemma 2.5]{CCG23})\label{partitions}
 For every $n\in\mathbb{N}$, the set $B_n$ is closed. In addition, for every $g,h\in G$ the following statements are equivalent:
 \begin{enumerate}
\item $gB_n\cap hB_n\neq \emptyset.$
\item $gB_n= hB_n.$
\item $g\Gamma_n=h\Gamma_n$.
 \end{enumerate}
Then the collection $\{g^{-1}B_n: g\in D_n^{-1}\}$ is a clopen partition of $X$, where $D_n$ is a subset of $G$ containing exactly one representative element of the classes  $\{\Gamma_ng: g\in G\}$, for every $n\in\NN$. 
\end{lemma}


\begin{proposition} \label{1-1-extension}
Let $x\in \Sigma^G$ be a Toeplitz element and let $X=\overline{O_G(x)}$.  Let $(\Gamma_n)_{n\in\mathbb{N}}$ be  a period structure of $x$, and  let $Y$ be the $G$-odometer associated to $(\Gamma_n)_{n\in\mathbb{N}}$. The map $\pi:X\to Y$ given by
$$
\pi(x)=(g_n\Gamma_n)_{n\in\mathbb{N}}, \mbox{ where  } x\in g_nB_n, \mbox{ for every } n\in\mathbb{N}
$$
is an almost 1-1 factor map. Since $(X,G)$ is minimal, this implies that   $(Y,G)$ is the maximal equicontinuous factor of $X$. 

Conversely, if $(Y,G)$ is a minimal equicontinuous Cantor system, then the symbolic almost 1-1 extensions of $(Y,G)$ are Toeplitz subshifts.
\end{proposition}
\begin{proof}
  The first part is direct from Lemma \ref{partitions}. The second part is a consequence of Proposition \ref{equicontinuous-characterization} and \cite[Theorem 5.8]{Kr10}.  
\end{proof}

 \section{Subsequences of finite index subgroups.}\label{Good-subsequences}




 Let $G$ be an infinite residually finite group.  Let $(Z,G)$ be a $G$-odometer associated to a sequence $(\Gamma_n)_{n\in\NN}$ of finite index subgroups of $G$. For every $n\in\NN$, let $C_n$ be the clopen set and $\P_n$ be the partition  introduced in Section \ref{odometers}. Let $(Y,G)$ be a minimal topological dynamical system such that there exists a factor map $\pi:Y\to Z$.   For every $y\in Y$ and $n\in\NN$, we define $t_n=t_{n,y}$ as the element in $D_n^{-1}$ such that
$$
\pi(y)\in t_{n,y}C_n.
$$
 
Observe  that $\pi(y)=(t_{n,y}\Gamma_n)_{n\in\NN}$. Thus, for every $g\in G$ and $n\in\NN$,  we have
\begin{equation}\label{1-1-eq}
gt_{n,y}=t_{n,gy}\gamma, \mbox{ for some } \gamma\in \Gamma_n.
\end{equation}

The goal of this section is to show the following result:

\begin{proposition}\label{suitable-sequence}
Let $G$ be an infinite residually finite group, and let $(Z,G)$ be  a minimal equicontinuous topological dynamical system with a free orbit.   Let $(Y,G)$ be a minimal topological dynamical system  such that there exists a factor map $\pi:Y\to Z$. Then  there exists a sequence $(\Gamma_n)_{n\in\NN}$ of finite index subgroups of $G$ such that $(Z,G)$ is conjugate to the $G$-odometer associated to $(\Gamma_n)_{n\in\NN}$ and such that  there  exists an increasing sequence 
$(D_n)_{n\in \mathbb{N}}$  of finite subsets of $G$,  such that 
\begin{itemize}
    \item[(P1)] $1_G\in D_1$ and $D_n\subseteq D_{n+1}$, for every $n\in \mathbb{N}$;
    \item[(P2)] $D_n$ contains exactly one element of each class $\{\Gamma_{n}g: g\in G\}$, for every $n\in \mathbb{N}$;
    \item[(P3)] $G=\bigcup_{n\in \mathbb{N}} D_n$;
   \item[(P4)] $D_{n+1}=\bigcup_{v\in D_{n+1}\cap \Gamma_{j}}vD_j$, for every $j\leq n$ and  $n\in \mathbb{N}$;
   \item[(P5)]  $\nu(Z_0)=0$, where
   $$
Z_0=\{(t_n\Gamma_n)_{n\in\NN}\in Z: G=\bigcup_{n\in\NN}t_n\Gamma_nD_{n-1}\}
$$
 and $\nu$ is the unique invariant probability measure of $(Z,G)$;
   \item[(P6)] For every $y\in Y$ and $n\in\NN$ there exists $h_n\in D_{n+1}\cap \Gamma_{n}t_{n,y}^{-1}$ such that 
$$h_n\in h_nD_{n} \subseteq D_{n+1}\setminus D_{n},$$ and 
$$
\lim_{n\to\infty}h_n^{-1}y_0=y,
$$
where $y_0$ is a fixed element in $\bigcap_{n\in\NN}\pi^{-1}C_n$;
\item[(P7)] $
\frac{|D_{n-1}|}{|D_n|}<\frac{1}{2^{n+1}}$,   for every  $n\in\NN$.
\end{itemize}
\end{proposition}

We need to prove some lemmas in order to show Proposition \ref{suitable-sequence}. With this purpose, in the rest of this section,  $(Z,G)$ and $(Y,G)$ are the dynamical systems given in Proposition \ref{suitable-sequence}. According to Lemma \ref{free-orbit} and Proposition \ref{equicontinuous-characterization}, the system $(Z,G)$ is conjugate to the $G$-odometer associated to a decreasing sequence $(\Gamma_n)_{n\in\NN}$ of finite index subgroups of $G$ with trivial intersection.

 It is not difficult to check that the hypothesis of normality of the finite index subgroups in \cite[Lemma 3]{CP14} and \cite[Lemma 2.8]{CCG23}  can be avoided. Thus

  we can assume that $(\Gamma_n)_{n\in\NN}$ is such that there exists a sequence   
$(D_i)_{i\in \mathbb{N}}$  of finite subsets of $G$ such that for every $i\in \mathbb{N}$
\begin{enumerate}
    \item $1_G\in D_1$ and $D_i\subseteq D_{i+1}$.
    \item $D_i$ contains exactly one element of each class $\{\Gamma_{i}g: g\in G\}$.
    \item $G=\bigcup_{i\in \mathbb{N}} D_i$.
   \item $D_{i+1}=\bigcup_{v\in D_{i+1}\cap \Gamma_{j}}vD_j$, for every $j\leq i$. 
 \end{enumerate}
\bigskip

Let $S=(n_i)_{i\in\NN}$ be an increasing sequence of positive integers. We define

$$
 Z_0(S)=\{(t_i\Gamma_i)_{i\in\NN}\in Z: G=\bigcup_{i\in\NN}t_{n_i}\Gamma_{n_i} D_{n_{i-1}}\}.
$$
This set is invariant with respect to the $G$ action on $Z$. Indeed, if $(t_i\Gamma_i)_{i\in\NN}\in Z_0(S)$ then given $h,g\in G$, there exists $i\in\NN$ such that $g^{-1}h\in t_{n_i}\Gamma_{n_i}D_{n_{i-1}}$. This implies that $h\in (gt_{n_i})\Gamma_{n_i}D_{n_{i-1}}$, from which we get that $g(t_i\Gamma_i)_{i\in\NN}\in Z_0(S)$.

\begin{lemma}\label{key-lemma}
If  $ S=(n_i)_{i\in\NN}$ is an increasing sequence of positive integers such that 
$$
\frac{|D_{n_{i-1}}|}{|D_{n_i}|}<\frac{1}{2^{i+1}}, \mbox{ for every } i\in\NN,
$$
then  $\nu(Z_0(S) )=0$.
 \end{lemma}

 \begin{proof}
Observe that $z=(t_i\Gamma_i)_{i\in\NN}\in Z_0(S)$ if and only if for every $g\in G$ there exists $i\in\NN$ such that $t_{n_i}^{-1}g\in \Gamma_{n_i}D_{n_i-1}$. 
This is equivalent to $g^{-1}t_{n_i}= v\gamma$, for some $\gamma\in \Gamma_{n_i}$ and $v\in D_{n_i-1}^{-1}$.
And the previous relation is satisfied if and only if 
 $g^{-1}z\in \bigcup_{d\in D_{n_{i-1}}^{-1}}dC_{n_i}$.
 This shows
\begin{equation}\label{2}
\bigcap_{g\in G}\bigcup_{j\in \NN}\bigcup_{d\in D_{n_{j-1}}^{-1}}gdC_{n_j} = Z_0(S). 
\end{equation}

 The relation (\ref{2}) implies that $Z_0(S)$ is an invariant measurable set. 

 
Let $A_j=\bigcup_{d\in D_{n_{j-1}}^{-1}}dC_{n_j}$. We have,
$$
\nu(gA_j)=\frac{|D_{n_{j-1}}|}{|D_{n_j}|}<\frac{1}{2^{j+1}}, \mbox{ for every } g\in G.
$$
Then
$$
\nu\left(\bigcup_{j\in\NN}gA_j\right)\leq \frac{1}{2}\sum_{j\in\NN}\frac{1}{2^j}=\frac{1}{2}.
$$
This implies
$$
\nu(Z_0(S))\leq \nu(\bigcap_{g\in G}\bigcup_{j\in\NN}gA_j)\leq \frac{1}{2},
$$
Since $Z_0(S)$ is a measurable invariant set and $\nu$ is ergodic, we deduce that $\nu(Z_0(S))=0$.
 \end{proof}

\begin{lemma}\label{subsequence2}
There exists a subsequence $(\Gamma_{n_i})_{i\in\NN}$ of   $(\Gamma_i)_{i\in\NN}$   for which
$$
\frac{|D_{n_{i-1}}|}{|D_{n_i}|}<\frac{1}{2^{i+1}}, \mbox{ for every } i\in\NN,
$$
and such that for every $y\in Y$ and $i\in\NN$ there exists $h_i\in D_{n_{i+1}}\cap \Gamma_{n_i}t_{n_i,y}^{-1}$ such that 
$$h_i\in h_iD_{n_i} \subseteq D_{n_{i+1}}\setminus D_{n_i},$$ and 
$$
\lim_{i\to\infty}h_i^{-1}y_0=y,
$$
where $y_0$ is an element in $\bigcap_{i\in\NN}C_i$.
 \end{lemma}
 
 \begin{proof}
 We define the subsequence $(\Gamma_{n_i})_{i\in\NN}$ as follows:

\bigskip

{\it Step 1: } Define $n_1\geq 1$ such that $|D_{n_1}|> 2$  and $k_1>n_1$ such that $d(y,y')\leq 2^{-k_1}$ implies $t_{n_1,y}=t_{n_1,y'}$. Let $\varepsilon_1=2^{-k_1}$. 

\bigskip

{\it Step 2:} Let $F\subseteq G$ be a finite set such that for every $y\in Y$ and every $g\in G$ there exists $h\in gF$ verifying $d(h^{-1}y_0,y)\leq \varepsilon_1$. Lemma \ref{minimality} ensures the existence of $F$.   
We take $n_2>n_1$ for which $|D_{n_2}|>2^{2}|D_{n_1}|$ and such that there exists $g\in G$  verifying $$ gFD_{n_1}\subseteq D_{n_2}\setminus D_{n_1}.$$ This is posible because $G$ is infinite. Given $y\in Y$ we can take $h_1\in gF$ such that  $d(h_1^{-1}y_0,y)\leq \varepsilon_1$. The choice of $\varepsilon_1$ implies that $h_1^{-1}\in t_{n_1,y}\Gamma_{n_1}$. The choice of $n_2$ implies
$$
h_1D_{n_1}\subseteq gFD_{n_1}\subseteq D_{n_2}\setminus D_{n_1}.
$$
Let $k_2>n_2$ be such that $d(y,y')\leq 2^{-k_2}$ implies $t_{n_2,y}=t_{n_2,y'}$. Let $\varepsilon_2=2^{-k_2}$.

\bigskip

{\it Step i+1:} Let $F\subseteq G$ be a finite set such that for every $y\in Y$ and every $g\in G$ there exists $h\in gF$ verifying $d(h^{-1}y_0,y)\leq \varepsilon_i$. Lemma \ref{minimality} ensures the existence of $F$. The choice of $\varepsilon_i$ implies that $h_i^{-1}\in t_{n_i,y}\Gamma_{n_i}$.  We take $n_{i+1}>n_i$ for which $|D_{n_{i+1}}|>2^{i+1}|D_{n_i}|$ and such that there exists $g\in G$ such that $$h\in gFD_{n_i}\subseteq D_{n_{i+1}}\setminus  D_{n_i}.$$ This is posible because $G$ is infinite. Given $y\in Y$ we can take $h_i\in gF$ such that  $d(h_i^{-1}y_0,y)\leq \varepsilon_i$. The choice of $\varepsilon_i$ implies that $h_i^{-1}\in t_{n_i,y}\Gamma_{n_i}$. The choice of $n_{i+1}$ implies
$$
h_iD_{n_i}\subseteq gFD_{n_i}\subseteq D_{n_2}\setminus D_{n_1}.
$$
Let $k_{i+1}>n_{i+1}$ be such that $d(y,y')\leq 2^{-k_{i+1}}$ implies $t_{n_{i+1},y}=t_{n_{i+1},y'}$.  Let $\varepsilon_{i+1}=2^{-k_{i+1}}$.
\end{proof}

\begin{proof}[Proof of Proposition \ref{suitable-sequence}]
  Proposition \ref{suitable-sequence} follows from Lemmas \cite[Lemma 3]{CP14}, \cite[Lemma 2.8]{CCG23}, \ref{key-lemma} and \ref{subsequence2}, and from the fact that $(Z,G)$ is conjugated to the $G$-odometer associated to $(\Gamma_{n_i})_{i\in\NN}$, for every increasing sequence $(n_i)_{i\in\NN}$.
  \end{proof}

 \section{Proof of Theorem \ref{theorem}.}\label{Section-borel-map}
  Let $G$ be an infinite residually finite group. In the rest of this section, $(Z,G)$ is a minimal equicontinuous Cantor system with a free orbit, and $(Y,G)$ is a minimal topological dy\-na\-mi\-cal system  such that there exists a factor map $\pi:Y\to Z$. Let $(\Gamma_n)_{n\in\NN}$ and $(D_n)_{n\in\NN}$ be sequences of finite index subgroups and finite sets of $G$, respectively, as in Proposition \ref{suitable-sequence}. For every $n\in\NN$, the clopen set $C_n\subseteq Z$ and the partition $\P_n$ of $Z$ are defined as in Section \ref{odometers}.    For every $y\in Y$, consider $t_{n,y}\in D_n^{-1}$ as in Section \ref{Good-subsequences}.
  
  Observe that for every $n\in\NN$,  the collection $\pi^{-1}\P_n=\{g\pi^{-1}C_n: g\in D_n^{-1}\}$ is a clopen partition of $Y$. Furthermore,
$g\pi^{-1}C_n\cap \pi^{-1}C_n\neq \emptyset$, if and only if $g\pi^{-1}C_n=\pi^{-1}C_n$, and if and only if $g\in\Gamma_n$.
  
  \begin{defi}
  For $z=(z_i\Gamma_i)_{i\in\NN}\in Z$ we define
  $$
  \Aper(z)=G\setminus \bigcup_{i\in\NN}z_{i}\Gamma_iD_{i-1}.
  $$
  and
  $$
  J_n(z)=z_n\Gamma_n\left(D_{n-1}\setminus \bigcup_{j=1}^{n-1}\Gamma_jD_{j-1}\right), \mbox{ for every } n\geq 1.
  $$
  If $y\in \pi^{-1}\{z\}$, then we set $\Aper(y)=\Aper(z)$ and $J_n(y)=J_n(z)$, for every $n\geq 1$.
  \end{defi}
  
  \begin{remark}
  Observe that $\Aper(y)=G\setminus \bigcup_{i\in\NN}t_{i,y}\Gamma_iD_{i-1}$. Thus, if   $y\in \bigcap_{i\in\NN}\pi^{-1}C_i$ then $\Aper(y)=G\setminus \bigcup_{i\in\NN}\Gamma_iD_{i-1}\subseteq G\setminus \bigcup_{i\in\NN}D_i= G\setminus G = \emptyset$.
  \end{remark}

\begin{lemma}\label{well-defined}
For every $z\in Z$,  the collection $\{J_n(z): n\geq 1\}$ is a partition of $G\setminus \Aper(z)$. Furthermore,
$$
g\in J_n(z) \mbox{ if and only if } n=\min\{j\geq 1: g\in z_j\Gamma_jD_{j-1}\}.
$$
\end{lemma}
\begin{proof}
  Let $g\in G\setminus \Aper(z)$, $i=\min\{j\geq 1: g\in z_j\Gamma_jD_{j-1}\}$ and $d_{i-1}\in D_{i-1}$ be such that $g\in z_i\Gamma_id_{i-1}$. Suppose there exist $1\leq j \leq i-1$ and $d_{j-1}\in D_{j-1}$  such that $d_{i-1}\in \Gamma_jd_{j-1}$. Then $g\in z_i\Gamma_id_{i-1}\subseteq z_j\Gamma_jd_{j-1}$, which is a contradiction with the choice of $i$. Then we deduce that $g\in J_i(z)$. 
 
Let $j>i$ be  such that there exists $d_{j-1}\in D_{j-1}$ such that 
  $g\in z_j\Gamma_jd_{j-1}$. This implies that $d_{j-1}\in \Gamma_id_{i-1}$, which shows that $d_{j-1}\notin D_{j-1}\setminus \bigcup_{l=1}^{j-1}\Gamma_jD_{j-1}$ and $g\notin J_j(z)$.   
\end{proof}

Since  $G=\bigcup_{i\in\NN}\Gamma_iD_{i-1}$, for every $g\in G$ we  define
$$
\min(g)=k \mbox{ such that }  g\in J_k(y_0).
$$
Observe that 
$$
\min(g)=\min(\gamma g), \mbox{ for every }    \gamma\in \Gamma_j \mbox{ and } j\geq \min(g).
$$

According to Lemma \ref{well-defined}, for every $z\in Z$ and $g\in G\setminus \Aper(z)$, there exist a unique $n\geq 1$ and $d_{n-1}\in D_{n-1}$ such that $g\in J_n(z)$ and $g\in z_n\Gamma_nd_{n-1}$.

Observe that equation (\ref{1-1-eq}) implies that
\begin{equation}\label{Jotas}
  gJ_n(y)=J_n(gy) \mbox{ for every } g\in G, y\in Y, n\in\NN.  
\end{equation}


\subsection{Construction of the Borel map.}
We consider $Y^G$ equipped with the product topology, and with the shift action given by
  $$
  g(y(h))_{h\in G}=(y(g^{-1}h))_{h\in G}, \mbox{ for every } g\in G \mbox{ and } (y(h))_{h\in G}\in Y^G.
  $$

Fix $y_0\in \bigcap_{i\in\NN}\pi^{-1}C_i$.  We define $\phi:Y\to Y^G$ as follows: for every $g\in G$ and $y\in Y$,

$$
 \phi(y)(g)=\left\{\begin{array}{ll}
                   d_{i-1}^{-1}y_0 & \mbox{ if  } g\in  J_i(y)   \mbox{ and  } d_{i-1}\in D_{i-1} \mbox{ is such that }\\
                     & g\in t_{i,y}\Gamma_id_{i-1}\\
                  g^{-1}y  & \mbox{ if } g\in \Aper(y). 
                  \end{array}\right.,
 $$
 The map $\phi$ is well defined, thanks to Lemma \ref{well-defined}.
\begin{remark}
Observe that if $\pi(y)=\pi(y')$, then $\phi(y)(g)=\phi(y')(g)$, for every $g\in G\setminus \Aper(y)$. Furthermore, if $z\in Z\setminus Z_0$ then $\phi$ is injective on $\pi^{-1}\{z\}$.
\end{remark}

\begin{proposition}\label{Borel-map}
The function $\psi:Y\to Z\times Y^G$ given by $\psi(y)=(\pi(y),\phi(y))$ for every $y\in Y$,   is measurable and equivariant. Furthermore, if 
$$
X=\overline{\{g\psi(y_0) :g \in G\}} \subseteq Z\times Y^G.
$$
then 
\begin{enumerate}
\item $(X,G)$ is a minimal almost one-to-one extension of $(Z,G)$ through the projection on the first coordinate $\tau:X\to Z$. In particular,   $$\tau^{-1}\{\pi(y_0)\}=\{\psi(y_0)\}.$$
\item $\psi(Y)\subseteq X$,
\item The system $(X,G)$ is Cantor if $(Y,G)$ is.
\end{enumerate}
\end{proposition}


\begin{proof} 
We will show that $\phi$ is a punctual limit of continuous and equivariant maps.
Let $i\geq 1$ and $y\in Y$. We define
\begin{align*}
   \phi_i(y)(h)=\begin{cases}
        d_{j-1}^{-1}y_0 &\mbox{ if  }  \exists j\leq i \mbox{ such that }  h\in J_j(y),  \mbox{ and } d_{j-1}\in D_{j-1} \mbox{ is such that }\\
             & h\in \Gamma_jd_{j-1} \\
        h^{-1}y,& \mbox{ otherwise.}
    \end{cases}
\end{align*}
It is clear that $\lim_{i\to\infty}\phi_i(y)=\phi(y)$.

 We will verify that $\phi_i$ is continuous.  Let $F\subseteq G$ be a finite set and $\varepsilon>0$. Let $F_1=F\cap \bigcup_{j=1}^{i}J_j(y)$.  Let $\delta>0$ be such that $d(y,y')<\delta$ implies $d(h^{-1}y,h^{-1}y')<\varepsilon$ for every $h\in F$. Let $y\in Y$. For every $y'\in \pi^{-1}C_i\cap B_{\delta}(y)$ we have $\phi_i(y)(h)=\phi_i(y')(h)$ for every $h\in F_1$ and $d(\phi_i(y)(h), \phi_i(y')(h))<\varepsilon$, for every $h\in F\setminus F_1$. This shows that $\phi_i$ is continuous. Since $\phi$ is a punctual limit of continuous functions, we deduce that $\phi$ is a Borel map.

 Now, we will check that $\phi_i$ is equivariant. Let $y\in Y$ and $g,h\in G$. From equation (\ref{Jotas}) we have 
 $$ h^{-1}g\in J_j(y) \mbox{ if and only if } g\in J_j(hy),$$ and in this case, $h^{-1}g\in t_{j,y}\Gamma_jd_{j-1}$ if and only if $g\in t_{i,hy}\Gamma_jd_{j-1}$, for $d_{j-1}\in D_{j-1}$. This implies that if $h^{-1}g\in J_j(y)$ for some $1\leq j\leq i$, then 
 $$
 h\phi_i(y)(g)=\phi_i(y)(h^{-1}g)=d_{j-1}^{-1}y_0=\phi_i(hy)(g) \mbox{ if } h^{-1}g\in J_j(y),  
 $$
and
$$
h\phi_i(y)(g)=\phi_i(y)(h^{-1}g)=g^{-1}hy=\phi_i(hy)(g) \mbox{ if not }.
$$
 This shows that $\phi_i$ is equivariant.

Since every $\phi_i$ is continuous and equivariant, we get that $\phi:Y\to Y^G$ is measurable and equivariant. This also implies that  $\psi: Y\to Z\times \phi(Y)\subseteq Z\times Y^G$ is equivariant and measurable.  

\bigskip

Observe that $\psi_i: Y\to \psi_i(Y)\subseteq Z\times Y^G$, given by $\psi_i(y)=(\pi(y), \phi_i(y))$, is a conjugacy.

\bigskip    
 
{\underline {\it Claim 1:} }
 For every $y\in Y$ there exists a sequence $(h_i)_{i\in\NN}$ in $G$    such that 
 $$
 \lim_{i\to \infty}h_i^{-1}y_0=y,
 $$
 and 
 $$
 \lim_{i\to \infty}h_i^{-1}\phi(y_0)=\phi(y).
 $$
 {\it Proof of Claim 1:}
 
 \bigskip
 

Let $y\in Y$ and  $g\notin \bigcup_{i\in\NN}t_{i,y}\Gamma_iD_{i-1}$. The property (P6) of $(\Gamma_i)_{i\in\NN}$ ensures the existence of a sequence $(h_i)_{i\in\NN}$ of elements of $G$, such that $h_i\in D_{i+1}\cap \Gamma_{i}t_{i,y}^{-1}$,
$$h_i\in h_iD_{i} \subseteq D_{i+1}\setminus D_{i},$$ and 
$$
\lim_{i\to\infty}h_i^{-1}y_0=y.
$$

 Let $i_0\in\NN$ be  such that $g\in D_{i_0}$. Thus, for every $i\geq i_0$ we have
$$
h_ig\in D_{i+1}\setminus D_i.
$$

If $h_ig\in \bigcup_{j=1}^{i+1}\Gamma_jD_{j-1}$, then $h_ig\in \bigcup_{j=1}^i\Gamma_jD_{j-1}$. In this case, there exists $1\leq j\leq i$ such that
$$
g\in h_i^{-1}\Gamma_jD_{j-1}\subseteq t_{j,y}\Gamma_jD_{j-1},
$$
which contradicts the hypothesis about $g$. Thus $h_ig\in D_{i+1}\setminus  \bigcup_{j=1}^{i+1}\Gamma_jD_{j-1}$, which implies that
$$
h_i^{-1}\phi(y_0)(g)=\phi(y_0)(h_ig)=g^{-1}h_i^{-1}y_0,
$$
and then
$$
\lim_{i\to \infty}h_i^{-1}\phi(y_0)(g)=g^{-1}y=\phi(y)(g).
$$
On the other hand, if $g\in G\setminus \Aper(y)$  let $k\in\NN$ be the smallest integer such that $g=t_{k,y}\gamma_k d_{k-1}$, for some $\gamma_k\in\Gamma_k$ and $d_{k-1}\in D_{k-1}$. Since $\lim_{i\to\infty}h_i^{-1}y_0=y$, for every sufficiently large $i$ we have $h_i^{-1}\in t_{k,y}\Gamma_k$,  which implies that $h_ig\in \Gamma_kd_{k-1}$, and $k$ the smallest integer with this property. Then we have
$$
\phi(y)(g)= d_{k-1}^{-1}y_0=  \phi(y_0)(h_ig)=h_i^{-1}\phi(y_0)(g),
$$
for every $i$ sufficiently large. This proves Claim 1.


\bigskip

 Claim 1  implies that for every $y\in Y$ there exists a sequence $(h_i)_{i\in\NN}$ of elements of $G$ such that  $\lim_{i\to\infty}h_i^{-1}(\pi(y_0),\phi(y_0))=(\pi(y), \phi(y))$. In other words, $\psi(Y)\subseteq X$.

 The map $\tau:  X \to Z$, given by the projection on the first coordinate, is almost one-to-one. Indeed, let $(z,x)\in X$ be such that $\tau(z,x)=\pi(y_0)$. Then we have $z=\pi(y_0)$. Let $(h_i(\pi(y_0),\phi(y_0)))_{i\in\NN}$ be a sequence converging to $(z,x)$. Let $g\in G$ and $k=\min(g)$. Since $\lim_{i\to \infty} h_i\pi(y_0)=\pi(y_0)$,   there exists $i_k\in\NN$ such that for every $i\geq i_k$ we have $h_i\in\Gamma_k$.  Thus $\min(h_i^{-1}g)=k$, which implies $h_i^{-1}\phi(y_0)(g)=\phi(y_0)(h_ig)=\phi(y_0)(g)=d_{k-1}^{-1}y_0$, where $d_{k-1}$ is the element in $D_{k-1}$ such that $g, h_ig\in \Gamma_kd_{k-1}$.  Since $G=\bigcup_{k\in\NN}\Gamma_kD_{k-1}$, we deduce that $\lim_{i\to\infty}h_i^{-1}\phi(y_0)=\phi(y_0)$, which implies that $(z,x)=(\pi(y_0),\phi(y_0))$.

 \medskip

 Since $X$ is equal to the closure of the orbit of $\psi(y_0)$, $(Z,G)$ is minimal and $\tau^{-1}\{\tau(\psi(y_0) \}=\{\psi(y_0)\}$, we have that $(X,G)$ is minimal (see, for example, \cite[Exercise 1.17]{Gl03}). Furthermore, if $Y$ is a Cantor set, then  $X\subseteq Z\times Y^G$  is totally disconnected. Since $\psi(Y)$ is infinite and $(X,G)$ is minimal,  Lemma \ref{Cantor} implies that $X$ is a Cantor set.

 \end{proof}

\subsection{The affine homeomorphism.}\label{Fibers}
Let $\psi:Y\to X$ be the map defined in Proposition \ref{Borel-map}. Consider $\psi^*: \M(Y)\to \M(X)$, given by $$\psi^*(\mu)(A)=\mu(\psi^{-1}(A)),$$ for every $\mu\in\M(Y)$ and every Borel set $X$. It is straightforward to check that $\psi^*$ is affine. Similarly, we define $\pi^*$ and $\tau^*$, where $\pi:Y\to Z$ and $\tau:X\to Z$ are the factor maps of Proposition \ref{Borel-map} (see Lemma \ref{sobre}).

The goal of this section is to show that $\psi^*$ is an affine bijection between $\M(Y,G)$ and $\M(X,G)$, and that when $Y$ is a Cantor set then $\psi^*$ is in addition continuous.

 \begin{lemma}\label{auxiliar-0}
  The map $\psi^*$ is one-to-one on the subset of measures of $\M(Y)$ which are supported on $\pi^{-1}(Z\setminus Z_0)$. In particular, $\psi^*$ is one-to-one on $\M(Y,G)$.   
 \end{lemma}
\begin{proof}
 Let $\mu_1$ and $\mu_2$ be measures in $\M(Y)$  which are supported on $\pi^{-1}(Z\setminus Z_0)$. If $\mu_1\neq \mu_2\in \M(Y)$  then there exists a Borel set $A\subseteq Y$ such that $\mu_1(A)\neq \mu_2(A)$. Since the measures are supported on $\pi^{-1}(Z\setminus Z_0)$, we get $\mu_1(A)=\mu_1(A\cap \pi^{-1}(Z\setminus Z_0))$ and $\mu_2(A)=\mu_2(A\cap \pi^{-1}(Z\setminus Z_0))$. On the other hand, $\psi(A\cap \pi^{-1}(Z\setminus Z_0))$ is a Borel set of $X$, because $\psi$ is one-to-one on $\pi^{-1}(Z\setminus Z_0)$   (see, for example, \cite[Theorem 2.8]{Gl03}). This implies that $\psi^*(\mu_1)(\psi(A\cap \pi^{-1}(Z\setminus Z_0)))=\mu_1(A)$ and $\psi^*(\mu_2)(\psi(A\cap \pi^{-1}(Z\setminus Z_0)))=\mu_2(A)$, which shows that $\psi^*$ is one-to-one.   Since every measure in $\M(Y,G)$ is supported on $\pi^{-1}(Z\setminus Z_0)$, we get the result.
\end{proof}

 \begin{lemma}\label{auxiliar-1}
 Let $g\in G$ and $m\in\NN$. Then
 $$ \pi^{-1}\{z\in Z: g\in J_m(z)\} \subseteq  g\left(\bigcup_{v\in D_{m-1}^{-1}}v\pi^{-1}C_m\right)   $$
 \end{lemma}
\begin{proof}
Observe that $g\in J_m(z)$ if and only if $1_G\in J_m(g^{-1}z)$. On the other hand, if $1_G\in J_m(g^{-1}z)$ then 
 $$
 g^{-1}z\in \bigcup_{v\in D_{m-1}^{-1}}vC_m.
 $$
\end{proof}

\begin{proposition}\label{continuity}
If $Y$ is a Cantor set, then the restriction of $\psi^*$ to $(\pi^*)^{-1}\{\nu\}$ is continuous. In particular, the restriction of $\psi^*$ to $\M(Y,G)$ is con\-ti\-nuous.
\end{proposition}
\begin{proof}
By the dominated convergence theorem,  $\psi^*$ is the punctual limit of the affine continuous maps $(\psi_i^*)_{i\in\NN}$. 
 
 Let $n < k\in\NN$ and $V_{k,n}=\{y\in Y: \phi(y)|_{D_n}\neq \phi_k(y)|_{D_n}\}$. Observe that
 $$
 V_{k,n}=\bigcup_{m > k}\bigcup_{g\in D_n}\pi^{-1}\{z\in Z: g\in J_m(z)\}.  
 $$
Since for every $g\in G$ and $m\in\NN$, the set $\{z\in Z: g\in J_m(z)\}$ is closed,  we have that $V_{k,n}$ is a Borel set. Furthermore, by Lemma \ref{auxiliar-1}, we have
$$
 V_{k,n}\subseteq \bigcup_{m >k}\bigcup_{g\in D_n}g \left(\bigcup_{v\in D_{m-1}^{-1}}v\pi^{-1}C_m\right).
 $$
 Thus, for every probability measure $\mu\in (\pi^*)^{-1}\{\nu\}$, we have
 $$
 \mu(V_{k,n})\leq |D_n|\sum_{m>k}\frac{|D_{m-1}|}{|D_m|}.
 $$
 From (P7) in Proposition \ref{suitable-sequence}, we get $\frac{|D_{m-1}|}{|D_m|}\leq \frac{1}{2^{m+1}}$, and then
 $$
 \mu(V_{k,n})\leq |D_n|\frac{1}{2^{k}}.
 $$
 Since $\phi(y)|_{D_n}=\phi_k(y)|_{D_n}$ implies that $\phi(y)|_{D_n}=\phi_l(y)|_{D_n}$, for every $l\geq k$, we have
 \begin{eqnarray*}
 \mu(\{y\in Y: \phi(y)|_{D_n}=\phi_l(y)|_{D_n} \mbox{ for every } l\geq k\}) &= & 1-\mu(V_{k,n}) \\
    &\geq &1-|D_n|\frac{1}{2^k}.
 \end{eqnarray*}


Let $\varepsilon >0$ and $n\in\NN$. Let $k_n>n$ be such that $|D_n|\frac{1}{2^{k_n}}<\frac{\varepsilon}{2}$. 
 
 Let $C\subseteq Z$ and $C_g\subseteq Y$ be open sets, for every $g\in D_n$. The set $$K=C\times \prod_{g\in D_n}C_g\times Y^{G\setminus D_n}$$ is an open subset of $Z\times Y^G$. We have
 $$
 \psi^{-1}K=\pi^{-1}C\cap\{y\in Y: \phi(y)(g)\in C_g, \mbox{ for every } g\in D_n\}
 $$
 and
 $$
 \psi_l^{-1}K=\pi^{-1}C\cap\{y\in Y: \phi_l(y)(g)\in C_g, \mbox{ for every } g\in D_n\}.
 $$
Observe that
$$
\psi^{-1}K\cap V_{k,n}^c=\psi_l^{-1}K\cap V_{k,n}^c, \mbox{ for every }l\geq k_n.
$$
This implies 
\begin{eqnarray*}
\mu(\psi^{-1}K) & = & \mu(\psi_l^{-1}K\cap V_{k,n}^c)+\mu(\psi^{-1}K\setminus V_{k,n}^c)\\
                & =  & \mu(\psi_l^{-1}K)-\mu(\psi_l^{-1}K\setminus V_{k,n}^c)+ \mu(\psi^{-1}K\setminus V_{k,n}^c)\\
\end{eqnarray*}
From this we get
\begin{equation}\label{metric-1}
|\mu(\psi^{-1}K)-\mu(\psi_l^{-1}K)|\leq \varepsilon \mbox{ for every } l\geq k_n.
\end{equation}

Let $(\mathcal{Q}_n)_{n\in\NN}$ be a sequence of nested clopen partitions of $Y$ that spans its topology. For every $n\in\NN$, we define
$$
\mathcal{K}_n=\{vC_n\times \prod_{g\in D_n}C_g\times Y^{G\setminus D_n}: v\in D_n^{-1}, C_g\in \mathcal{Q}_n, \mbox{ for every } g\in D_n\}.
$$
The collection $(\mathcal{K}_n)_{n\in\NN}$ is a sequence of nested clopen partitions that spans the topology of $Z\times Y^G$.

For every $\mu$ and $\mu'$ in $\M(Z\times Y^G)$, consider
$$
d(\mu,\mu')=\sum_{n=1}^{\infty}\sum_{K\in \mathcal{K}_n}\frac{1}{2^ns_n}|\mu(K)-\mu'(K)|, 
$$
where $s_n$ is the number of elements of $\mathcal{K}_n$, for every $n\in\NN$. This is the metric defined in (\ref{newmetric}), which generates the weak topology on $\M(Z\times Y^G)$ (see Corollary \ref{Coro24}).


Let $\varepsilon>0$ and  $m\in\NN$ be such that $\sum_{n>m}\frac{1}{2^n}<\frac{\varepsilon}{2}$. From equation (\ref{metric-1}), there exists $k_m\in\NN$ such that for every $l\geq k_m$ and every measure $\mu\in (\pi^*)^{-1}\{\nu\}$, we have
$$
|\mu(\psi^{-1}K)-\mu(\psi_l^{-1}K)|\leq \frac{\varepsilon}{2}\left(\sum_{n=1}^m\frac{1}{2^n} \right)^{-1} \mbox{ for every  } K\in \mathcal{K}_j, 1\leq j\leq m.
$$
This implies that
\begin{align}\label{metric-2}
d(\psi^*\mu, \psi^*_l\mu)\leq \frac{\varepsilon}{2}+\sum_{n>m}\frac{1}{2^n}<\varepsilon. 
\end{align}

Let $(\mu_i)_{i\in\NN}$ be a sequence of measures in $(\pi^*)^{-1}\{\nu\}$  that weakly converges to $\mu \in (\pi^*)^{-1}\{\nu\}$. We have
\begin{eqnarray}\label{eq-n}
d(\psi^*(\mu_i),\psi^*(\mu)) &\leq & d(\psi^*(\mu_i),\psi_k^*(\mu_i))+ d(\psi_k^*(\mu_i), \psi_k^*(\mu))+ d(\psi_k^*(\mu),\psi^*(\mu)), 
\end{eqnarray}
for every $i, k\in\NN$. Equation (\ref{metric-2}) implies there exists $k\in\NN$ such that 
\begin{equation}\label{eq-nn}
d(\psi^*(\mu_i),\psi_k^*(\mu_i))<\frac{\varepsilon}{3} \hspace{5mm} \mbox{ and } \hspace{5mm} d(\psi_k^*(\mu),\psi^*(\mu)) <\frac{\varepsilon}{3}
\end{equation}

By the continuity of $\psi^*_{k}$,  the sequence  $(\psi^*_{k}(\mu_i))_{i\in\NN}$ weakly converges to $\psi^*_k(\mu)$.  Thus, there exists $i_0\in\NN$ such that    
\begin{equation}\label{eq-nnn}
d(\psi^*_{k}(\mu_i), \psi^*_{k}(\mu))<\frac{\varepsilon}{3} \mbox{ for every } i\geq i_0.
\end{equation}
By equations (\ref{eq-n}), (\ref{eq-nn}) and (\ref{eq-nnn}), we get 
$$
d(\psi^*(\mu_i), \psi^*(\mu))<\varepsilon \mbox{ for every } i\geq i_0
$$
This shows that the restriction of $\psi^*$ to $(\pi^*)^{-1}\{\nu\}$ is continuous. Since $\M(Y,G)\subseteq (\pi^*)^{-1}\{\nu\}$, in particular, the restriction of $\psi^*$ to $\M(Y,G)$ is continuous.
\end{proof}

\begin{lemma} \label{simple}
Let $z\in Z$ and $(z,x)\in \tau^{-1}\{z\}$.
\begin{enumerate}
\item If $g\in G\setminus \Aper(z)$ and $k\geq 1$ is such that $g\in J_k(z)$, then   
$$x(g)=d_{k-1}^{-1}y_0, 
$$
where  $d_{k-1}\in D_{k-1}$ is such that $g\in z_k\Gamma_kd_{k-1}$.
\item If $g\in \Aper(z)$, then
$$
\pi(x(g))=g^{-1}z.
$$
\end{enumerate}
\end{lemma}
\begin{proof}
Let $(h_i(\pi(y_0), \phi(y_0)))_{i\in\NN}$ be a sequence converging to $(z,x)$. For every $k\in \NN$, there exists $i_k\in\NN$ such that $h_i\in z_k\Gamma_k$, for every $i\geq i_k$.

\medskip

1. Suppose that $g\in G\setminus \Aper(z)$. Let $k=\min\{j\geq 1: g\in z_j\Gamma_jD_{j-1}\}$ and $d_{k-1}\in D_{k-1}$ be such that $g\in z_k\Gamma_kd_{k-1}$. 
 This implies that for every $i\geq i_k$, 
$$
h_i\phi(y_0)(g)=\phi(y_0)(h_i^{-1}g)=d_{k-1}^{-1}y_0,
$$
which shows that $x(g)=d_{k-1}^{-1}y_0.$

\medskip 

2. Suppose that $g\in \Aper(z)$. For every $i\in\NN$, let $\min(h_i^{-1}g)=l_i$ and $d_{l_i-1}\in D_{l_i-1}$ be such that $h_i^{-1}g\in \Gamma_{l_i}d_{l_i-1}$. We have $h_iy_0(g)=d_{l_i-1}^{-1}y_0$. Since $g\in \Aper(z)$, for every $k\in \NN$ there exists $j_k\in\NN$ such that $l_i>k$, for each $i\geq j_k$. Thus,  for every $i\geq \max\{i_k,j_k\}$ we have
$$h_i\phi(y_0)(g)=d_{l_i-1}^{-1}y_0\in (g^{-1}z_k\Gamma_k\Gamma_{l_i})y_0=(g^{-1}z_k\Gamma_k)y_0,$$
which implies that 
$$
\pi(x(g))=\lim_{i\to\infty}h_i\pi(\phi(y_0)(g))=g^{-1}z.
$$
\end{proof}

\begin{remark}
 {\rm Lemma \ref{simple} implies that the elements in $\tau^{-1}\{z\}$ coincide on the coordinates $g\in G\setminus \Aper(z)$.  }   
\end{remark}

\begin{lemma}\label{convergencia}
Let $\mu$ be an invariant probability measure of $(X,G)$. Then there exists a sequence $(\mu_n)_{n\in\NN}$ of probability measures in $\M(Y)$ such that
\begin{enumerate}
\item There exists $\omega \in (\pi^*)^{-1}\{\nu\}\subseteq \M(Y)$ such that  $\lim_{n\to\infty}\mu_n=\omega$,  
\item $\lim_{n\to\infty}\psi^*(\mu_n)=\mu.$
\end{enumerate}
\end{lemma}
\begin{proof}
   Let $x\in X$ and $(y_i)_{i\in\NN}$ be a sequence of elements in $Y$ such that $x =\lim_{i\to\infty}\psi(y_i)$. Since the map from $X$ to $\M(X)$, that identifies every point with its respective Dirac measure, is continuous, we have that $\lim_{i\to\infty}\delta_{\psi(y_i)}=\delta_x$. It is easy to verify that $\delta_{\psi(y)}=\psi^*(\delta_{y})$, for every $y\in Y$. From this we get that $\delta_x\in \overline{\psi^*(\M(Y))}$. Since $\overline{\psi^{*}(\M(Y))}$ is convex and the set of convex combination of the elements in $\{ \delta_x: x\in X\}$ is  
dense in $\M(X)$,   we have $\M(X)\subseteq \overline{\psi^*(\M(Y))}$.  This implies there exists a sequence $(\lambda_n)_{n\in\NN}$  of elements in $\M(Y)$ such that $\lim_{n\to\infty}\psi^*(\lambda_n)=\mu$.

For every $n, i\in\NN$, consider the measure
$$
\mu_{n,i}=\frac{1}{|D_n|}\sum_{g\in D_n}g\lambda_i,
$$
where $g\lambda_i$ is the measure given by $g\lambda_i(A)=\lambda_i(g^{-1}A)$, for every Borel set $A\subseteq Y$. Since $hC_n=gC_n$ for every $h\in g\Gamma_n$, we have 
$$ 
\mu_{n,i}(\pi^{-1}(gC_n))=\mu_{n,i}(\pi^{-1}(C_n)), \mbox{ for every } g\in G.
$$
 
Moreover, 
\begin{equation}\label{convergencia-1}
\mu_{n,i}(\pi^{-1}(C_m))=\mu_{n,i}(\pi^{-1}(gC_m))  \mbox{ for each } m\leq n \mbox{ and } g\in G.
\end{equation}
Indeed,  $C_m$ is a disjoint union of translation of $C_n$, more precisely, $C_m=\bigcup_{\gamma\in \Gamma_m\cap D_n}\gamma C_n$. 
Observe that if $d,d'\in D_n$ are such that $gdC_n=gd'C_n$, then $d=d'$. This implies
 $\{g\gamma C_n\colon \gamma\in \Gamma_m\cap D_n\}$ is a collection of disjoint sets.
Thus, 
\begin{align*}
    \mu_{n,i}(\pi^{-1}(gC_m))&=\sum_{\gamma\in \Gamma_m\cap D_n}\mu_{n,i}(\pi^{-1}(g\gamma C_n))=|\Gamma_m\cap D_n|\mu_{n,i}(\pi^{-1}(C_n))\\
    &=\sum_{\gamma\in \Gamma_m\cap D_n}\mu_{n,i}(\pi^{-1}(\gamma C_n))=\mu_{n,i}(\pi^{-1}(C_m)).
\end{align*}

Furthermore, since $\mu$ is invariant and the action of $G$ on $\M(Y)$ is continuous, we get
$$
\lim_{i\to \infty}\psi^*(\mu_{n,i})=\mu.
$$
 

Let $(\varepsilon_n)_{n\in\NN}$ be a decreasing sequence of positive numbers going to $0$. Let $(i_n)_{n\in\NN}$ be an increasing sequence of positive integers such that
 
$$
d(\psi^*(\mu_{n,i_n}), \mu)<\varepsilon_n.
$$

Let $\omega\in \M(Y)$ be an accumulation point of $(\mu_{n,i_n})_{n\in\NN}$. Equation (\ref{convergencia-1}) implies that for every $g\in G$ and $n\in\NN$,
$$
\omega(\pi^{-1}(gC_n))=\omega(\pi^{-1}(C_n)),
$$
in other words, $\omega \in (\pi^*)^{-1}\{\nu\}$. Taking $(\mu_n)_{n\in\NN}$ as a subsequence of $(\mu_{n,i_n})_{n\in\NN}$ that converges to $\omega$, we get the result.
\end{proof}

\begin{proposition}\label{supported}
  If $\mu$ is an invariant probability measure of $(X,G)$, then $\mu$ is supported on $\psi(Y)$.  
\end{proposition}

\begin{proof}
By Lemma \ref{simple}, if $(z,x)\in X\setminus \psi(Y)$, then there exist $g,h\in \Aper(z)$ and $y_1\neq y_2\in \pi^{-1}\{z\}$ such that $x(g)=g^{-1}y_1$ and $x(h)=h^{-1}y_2$. Let $U_1$ and $U_2$ be disjoint open neighbourhoods of $y_1$ and $y_2$ respectively. We denote
$$
W_{z_k,g,h,x}=z_kC_k\times g^{-1}U_1\times h^{-1}U_2\times Y^{G\setminus\{g,h\}},
$$
where $z=(z_i\Gamma_i)_{i\in\NN}$ and $k\in\NN$. We have $(z,x)\in W_{z_k,g,h,x}$. If $y\in \psi^{-1}(W_{z_k,g,h,x})$ then $g\in G\setminus \Aper(y)$ or $h\in G\setminus \Aper(y)$. Furthermore, since $g,h\notin \bigcup_{i=1}^kz_i\Gamma_iD_{i-1}$, we have that $g$ or $h$ are in $J_m(y)$, for some $m>k$. Thus, from Lemma \ref{auxiliar-1} we have
$$
\psi^{-1}(W_{z_k,g,h,x})\subseteq \bigcup_{m>k}\bigcup_{t\in \{g,h\}}t\left(\bigcup_{v\in D_{m-1}^{-1}}v\pi^{-1}C_m\right).
$$
Observe this implies the following: for every $n\in \NN$, let define $X_n$ the set of the
$(z,x)\in X$ for which   there exists   $g,h\in D_n$  such that $x(g)$  and  $x(h)$  are in different orbits. Then $X_n$ is contained in an open set $W_{k,n}\subseteq Z\times Y^G$, which is the union of sets of the kind $W_{z_k,g,h,x}$, for $z_k\in D_k^{-1}$, $x\in X_n\cap \tau^{-1}(z_kC_k)$ and  $g\neq h \in D_n$. Furthermore, we have
$$
\psi^{-1}(W_{k,n})\subseteq \bigcup_{m >k}\bigcup_{g\in D_n}g \left(\bigcup_{v\in D_{m-1}^{-1}}v\pi^{-1}C_m\right).
$$
Let $(\mu_i)_{i\in\NN}$ be the sequence of measures in $\M(Y)$  of Lemma \ref{convergencia} associated to the invariant measure $\mu\in\M(X,G)$. Let $\omega\in(\pi^*)^{-1}\{\nu\}$ be such that $\lim_{i\to\infty}\mu_i=\omega$ (see (1) of Lemma \ref{convergencia}). Since $W_{k,n}$ is an open set in $Z\times Y^G$, we have
\begin{eqnarray*} 
\mu(W_{k,n}) & \leq & \liminf_i \psi^*(\mu_i)(W_{k,n})\\
         & = & \liminf_i \mu_i(\psi^{-1}(W_{k,n}))\\
         &\leq & \liminf_i \mu_i\left( \bigcup_{m >k}\bigcup_{g\in D_n}g \left(\bigcup_{v\in D_{m-1}^{-1}}v\pi^{-1}C_m \right)\right)\\
         &\leq & \liminf_i\sum_{m>k}\sum_{g\in D_n}\mu_i\left( g \left(\bigcup_{v\in D_{m-1}^{-1}}v\pi^{-1}C_m \right)   \right)\\
         &\leq & \limsup_i\sum_{m>k}\sum_{g\in D_n}\mu_i\left( g \left(\bigcup_{v\in D_{m-1}^{-1}}v\pi^{-1}C_m \right)   \right)\\
         &\leq & \sum_{m>k}\sum_{g\in D_n}\limsup_i\mu_i\left( g \left(\bigcup_{v\in D_{m-1}^{-1}}v\pi^{-1}C_m \right)\right)
\end{eqnarray*}
Since $g \left(\bigcup_{v\in D_{m-1}^{-1}}v\pi^{-1}C_m\right)$ is a clopen set,  point (1) of Lemma \ref{convergencia} implies that
 \begin{eqnarray*}
 \limsup_i\mu_i\left( g \left(\bigcup_{v\in D_{m-1}^{-1}}v\pi^{-1}C_m \right)\right) & = & \lim_{i\to\infty}\mu_i\left( g \left(\bigcup_{v\in D_{m-1}^{-1}}v\pi^{-1}C_m \right)\right)\\
 &=& \omega\left(  \bigcup_{v\in D_{m-1}^{-1}}gv\pi^{-1}C_m  \right)\\
 &\leq & \frac{|D_{m-1}|}{|D_m|},
\end{eqnarray*}
because $\omega \in (\pi^*)^{-1}\{\nu\}$. Thus we get
$$
\mu(W_{k,n})\leq \sum_{m>k}|D_n|\frac{|D_{m-1}|}{|D_m|}\leq |D_n|\sum_{m>k}\frac{1}{2^m}.
$$
This implies that for every $\varepsilon>0$, there exists $k\in\NN$ such that $\mu(W_{k,n})\leq \varepsilon$. Since $X_n$ is contained in $W_{k,n}$, for every $k\in\NN$, we deduce that $\mu(X_n)=0$. Finally, due to $X\setminus \psi(Y)=\bigcup_{n\in\NN}X_n$, we conclude that $\mu$ is supported on $\psi(Y)$.
\end{proof}

\begin{proposition}\label{measures}
The map $\psi^*:\M(Y,G)\to \M(X,G)$ is an affine bijection. If in addition, $Y$ is a Cantor set, then $\psi^*$ is a homeomorphism.   
\end{proposition}
\begin{proof}
Lemma \ref{auxiliar-0} and Proposition \ref{supported} imply that $\psi^*$ is an affine bijection between $\M(Y,G)$ and $\M(X,G)$. Proposition \ref{continuity} implies that $\psi^*$ is a homeomorphism when $Y$ (and then $X$) is a Cantor set.
\end{proof}

\begin{proof}[Proof of Theorem \ref{theorem}]
 This a consequence of Propositions \ref{Borel-map}  and \ref{measures}. 
\end{proof}

\bigskip

 \section{Proof of Theorem \ref{Main-Toeplitz-measures}}\label{General-case}

 The goal of this section is to prove Theorem \ref{Main-Toeplitz-measures} and Corollary \ref{corollary}. 
 For that purpose, we will prove the following  proposition:

 \begin{proposition}\label{main-tool}
 Let $G$ be a non-amenable countable group. Let  $(Y,G)$ be a minimal topological dynamical system  verifying the following properties:
\begin{enumerate}
\item There exists a minimal Cantor system $(X,G)$ without invariant pro\-ba\-bi\-li\-ty measures,  that is an almost one-to-one extension of $(Y,G)$.
\item There exists a minimal subshift $(Z,G)$ that is an almost one-to-one extension of $(Y,G)$.
\end{enumerate}
 Then there exists a minimal subshift $(\tilde{Z}, G)$, with no invariant probability measures, that is an almost one-to-one extension of $(Y,G)$. 
\end{proposition}


\begin{lemma}\label{subshift-factor}
 Let $(X, G)$ be a  Cantor system with no invariant probability measures. Then there exists a subshift $(Z,G)$  with no invariant probability measures which is a topological factor of $(X,G)$.
 \end{lemma}
\begin{proof}
  Let $(\P_n)_{n\in\NN}$ be a nested sequence of finite clopen partitions of $X$ such that\\ $\{A\in \P_n: n\in\NN\}$ spans the topology of $X$. 
For every clopen set $A\subseteq X$, let define
$$
Z_A=\left \{\mu \in \M(X): \mu(A)=\mu(gA) \mbox{ for every } g\in G\right\}. 
$$
We claim there exists $n\in\NN$ such that $\bigcap_{A\in\P_n}Z_A=\emptyset$. Suppose this is not true, then for every $n\in \NN$ there exists $\mu_n\in \bigcap_{A\in\P_n}Z_A$. Let $(\mu_{n_i})_{i\in\NN}$ be a subsequence that converges to $\mu\in\M(X)$. Given a clopen set $A\subseteq X$ there exists $i_0\in\NN$ such that for every $i\geq i_0$, there exist $A_{i,1},\cdots, A_{i,k_i}$ in $\P_{n_i}$ such that $A=\bigcup_{j=1}^{k_i}A_{i,j}$. This implies that for every $i\geq i_0$, and every $g\in G$, 
$$\mu_{n_i}(A)=\sum_{j=1}^{k_i}\mu_{n_i}(A_{i,j})=\sum_{j=1}^{k_i}\mu_{n_i}(gA_{i,j})=\mu_{n_i}(gA).$$
Since $A$ and $gA$ are clopen, we have $\mu(A)=\lim_{i\to \infty}\mu_{n_i}(A)$ and $\mu(gA)=\lim_{i\to \infty}\mu_{n_i}(gA)$, which implies that $\mu$ is invariant, contradicting that $\M(X,G)=\emptyset$. This shows the claim.

 Let $n\in\NN$ be such that $\bigcap_{A\in\P_n}Z_A=\emptyset$. Let $\P_n=\{A_1,\cdots, A_k\}$ and $\Sigma=\{1,\cdots, k\}$.
 Let $\pi: X\to \Sigma^G$ be the function defined as
$$
\pi(x)=(x(g))_{g\in G}, \mbox{ where } x(g)=i \mbox{ if and only if } g^{-1}x\in A_i. 
$$ 
The map $\pi$ is continuous and equivariant with respect to the shift action on $\Sigma^G$. This implies that $\pi:X\to \pi(X)$ is a factor map. The set $Z=\pi(X)$ is a subshift of $\Sigma^G$. The factor map $\pi$ induces an affine surjective and continuous transformation  $\pi^*:\M(X)\to \M(Z)$ given by $\pi^*\mu(A)=\mu(\pi^{-1}(A))$, for every Borel set $A\subseteq Z$ (see Lemma \ref{sobre}).  Suppose there exists  $\mu_0\in \M(Z, G)$ and  let $\mu\in (\pi^*)^{-1}\{\mu_0\}$. For every $1\leq i\leq k$ let $C_i=\{x\in Z: x(1_G)=i\}$. For every $g\in G$ we have $gA_i=g\pi^{-1}C_i=\pi^{-1}gC_i$, which implies
$$
\mu(gA_i)=\mu(\pi^{-1}gC_i)=\mu_0(gC_i)=\mu_0(C_i)=\mu(\pi^{-1}C_i)=\mu(A_i).
$$
From this, we get that $\mu\in \bigcap_{A\in\P_n}Z_A$, which is a contradiction.  This shows that $\M(Z,G)$ has no invariant probability measures.
 \end{proof}



The proof of Theorem \ref{main-tool}, will be the consequence of a series of results that we show below.

\begin{lemma}\label{sub-ext-1-1}
Let $G$ be a non-amenable countable group. Let $(Y,G)$ be a minimal topological dynamical system verifying the following properties:
\begin{enumerate}
 \item There exist a minimal Cantor system $(X,G)$, with no invariant probability measures,  and an   almost one-to-one factor map $\pi_X:X\to Y$.
 \item There exist a minimal subshift $(Z,G)$, and factor maps $\pi_Z:Z\to Y$ and $\pi:X\to Z$.
 \item $\pi_X=\pi_Z\circ \pi$.
\end{enumerate}
Then there exists a minimal subshift $(\tilde{Z},G)$, with no invariant probability measures, that is an almost 1-1 extension of $(Y,G)$.
\end{lemma}

\begin{proof}
From Lemma \ref{subshift-factor}, there exist a minimal subshift $(W,G)$, with no invariant probability measures, and a topological factor map $\phi:X\to W$. Let $\tilde{\pi}:X\to W\times Z$ be the map given by $\tilde{\pi}(x)=(\phi(x), \pi(x))$, for every $x\in X$.
The map $\phi$ is continuous and equivariant, which implies that $\tilde{\pi}:X\to \tilde{Z}=\tilde{\pi}(X)\subseteq W\times Z$ is a factor map.  
Observe that $(\tilde{Z}, G)$ is a minimal subshift of $(\Sigma_W\times \Sigma_Z)^G$, where $\Sigma_W$ and $\Sigma_Z$ are the finite alphabets on which $W$ and $Z$ are defined, respectively.
Let $\tau_W:\tilde{Z}\to W$ and $\tau_Z: \tilde{Z}\to Z$ the respective projection maps. 
These maps are continuous, equivariant, and by the minimality of $W$ and $Z$, they satisfy  $\tau_W(\tilde{Z})=W$ and $\tau_Z(\tilde{Z})=Z$.
Since $(\tilde{Z}, G)$ is an extension of $(W,G)$, the space  of invariant probability measures of $(\tilde{Z},G)$ is empty. 
On the other hand,  $\tau: \tilde{Z}\to Y$ defined by $\tau=\pi_Z\circ \tau_Z$, is a factor map.
By hypothesis, there exists $y_0\in Y$ such that $|\pi_X^{-1}\{y_0\}|=1$. Let $z_1$ and $z_2$ in $\tilde{Z}$ such that $\tau(z_1)=\tau(z_2)=y_0$.
Let $x_1, x_2\in X$ be such that $z_1=(\phi(x_1), \pi(x_1))$ and $z_2=(\phi(x_2), \pi(x_2))$. 
We have $\tau(z_1)=\pi_Z(\pi(x_1))=\pi_X(x_1)=y_0$, which implies that $x_1\in \pi_X^{-1}\{y_0\}$.
Similarly, we get $x_2\in \pi_X^{-1}\{y_0\}$, which implies $x_1=x_2$ and then $z_1=z_2$.
This shows that $(\tilde{Z}, G)$ is an almost 1-1 extension of $(Y,G)$. 
\begin{figure}
\begin{align}\label{figure1}
    \xymatrix{&\tilde{Z}\ar[dl]_{\tau_W}\ar^{\tau_Z}[dr]&\\
    W&\ar[l]_{\phi} X \ar[u]_{\tilde{\pi}} \ar[r]^{\pi} \ar[d]_{\pi_X}& Z\ar[dl]^{\pi_Z}\\
&Y&
}
\end{align}
\caption{Diagram corresponding to the different maps  of Lemma \ref{sub-ext-1-1}.}
\end{figure}
\end{proof}

The following lemma is a well known result.
Nevertheless, we add a proof for completeness.
\begin{lemma}\label{2-comp-almost}
Let $(X,G),(Y,G)$ and $(Z,G)$ be topological dynamical systems.  Suppose that $(Y,G)$ minimal.
    If $\pi: X\to Y$ and $\phi:Y\to Z$ are two almost $1$-$1$ factor maps, then $\phi\circ \pi:X\to Z$ is an almost $1$-$1$ factor map. 
\end{lemma}
\begin{proof} Let $S=\{y\in Y : |\pi^{-1}(y)|=1\}$ and $T=\{z\in Z : |\phi^{-1}(z)|=1\}$ be the residual sets given by the almost $1$-$1$ factor maps $\pi$ and $\phi$, respectively.
Consider $y\in \phi^{-1}(T)$. 
Since $T$ is $G$-invariant and $\phi$ is a factor map, the orbit of $y$ is contained in  $\phi^{-1}(T)$.
Thus the minimality of $Y$ implies that $\phi^{-1}(T)$ is a $G_\delta$ dense set in $Y$.
By Baire's category theorem, we have $\phi^{-1}(T)\cap S\neq \emptyset$.
Take $y\in \phi^{-1}(T)\cap S$. Let $x$ and $x'$ be two elements of $X$ such that $\phi(\pi(x))=\phi(\pi(x'))=\phi(y)$.
Since $\phi(y)\in T$, we have $\pi(x)=\pi(x')=y\in S$. Consequently, $x=x'$, and since $Z$ is minimal, we conclude that $\phi\circ\pi$ is an almost $1$-$1$ factor map.
\end{proof}

\begin{lemma}\label{inverse-continuous} Let $(X,G)$ and $(Y,G)$ be two topological dynamical systems. Suppose that  $\pi:X\to Y$ is an almost 1-1 factor map. Let $T\subseteq \{y\in Y: |\pi^{-1}(y)|=1\}$.
Then the map $\hat{\pi}:T\to X$, given by $\hat{\pi}(y)=z$, where $\{z\}=\pi^{-1}\{y\}$, is continuous, injective and $\hat{\pi}^{-1}|_{\hat{\pi}(T)}$   is continuous.
\end{lemma}
\begin{proof}  
     Let $C\subseteq X$ be a closed set. Since $\pi$ is a continuous map between compact Hausdorff spaces, the set   $\pi(C)$ is closed.
     Let us show that $\pi(C)\cap T=\hat{\pi}^{-1}(C)$. If $y \in \pi(C)\cap T$, then there exists a unique $z \in \pi^{-1}\{y\}$  with $z\in C$. Since $\hat{\pi}(y)=z \in C$, we get  $\pi(C)\cap T\subseteq \hat{\pi}^{-1}(C)$.
     On the other hand, if $y \in \hat{\pi}^{-1}(C)$, then $y\in T$ and $\hat{\pi}(y)=\pi^{-1}(y)=z\in C$, which implies that $y\in \pi(C)\cap T\subseteq \pi(C)$. Thus we conclude that $\pi(C)\cap T=\hat{\pi}^{-1}(C)$. This shows that the pre-image by  $\hat{\pi}$ of every closed set   is closed in $T$, which implies that $\hat{\pi}$ is continuous. The injectivity follows from the fact that $T\subseteq \{y\in Y:|\pi^{-1}(y)|=1\}$.  Since $\hat{\pi}^{-1}|_{\hat{\pi}(T)}=\pi|_{\hat{\pi}(T)}$, the map $\hat{\pi}^{-1}|_{\hat{\pi}(T)}$ is continuous.
\end{proof}

\begin{lemma}\label{symbolic-factor} Let  $(Y,G)$ be a minimal topological dynamical system verifying the following properties:
\begin{enumerate}
\item There exist a minimal Cantor system $(X,G)$ and an almost one-to-one factor map $\phi:X\to Y$.
\item There exist a subshift $(Z,G)$ and an almost one-to-one factor map $\pi_Z:Z\to Y$.
\end{enumerate}
Then there exist
\begin{itemize}
 \item A minimal Cantor system $(\tilde{X},G)$ which is a topological extension of $(X,G)$.  
 \item A factor map $\pi:\tilde{X}\to Z$ and an almost one-to-one factor map $\pi_{\tilde{X}}:\tilde{X}\to Y$
\end{itemize}
 such that $$\pi_{\tilde{X}}=\pi_Z\circ \pi.$$
\end{lemma}
\begin{proof}
 Denote $\mathcal{T}=\{y\in Y: |\pi_Z^{-1}\{y\}|=1\}$,  and $H=\{y\in Y: |\phi^{-1}\{y\}|=1\}$ the $G$-invariant residual sets given by the almost $1$-$1$ factor maps $\pi_Z$ and $\phi$, respectively. 
Note that $T=\mathcal{T}\cap H$ is still a $G$-invariant residual set, thanks to the Baire's category theorem, and consequently, a non-empty set.
Consider the maps $\hat{\pi_Z}: T\to Z$ and $\hat{\phi}:T\to X$, such that  $\hat{\pi_Z}(y)$ and $\hat{\phi}(y)$ are the unique elements in $\pi_Z^{-1}\{y\}$ and  $\phi^{-1}\{y\}$, respectively, for every $y\in T$.    Lemma \ref{inverse-continuous} implies that these maps are continuous, injective  with continuous inverse maps.
Denote $X'=\phi^{-1}(T)$, and let define
 $\tilde{X}=\overline{\{(x,\hat{\pi_Z}(\phi (x))):x\in X'\}}\subset X\times Z$.
   The group $G$ acts on $X\times Z$ coordinatewise. The set $\tilde{X}$ is invariant with respect to this action, because is the closure of an invariant set. 
 Let $\tau_X: \tilde{X}\to X$ and $\pi:\tilde{X}\to Z$ the respective projection maps.  These maps are equivariant, and the minimality of $X$ and $Z$ imply that $\tau_X$ and $\pi$ are surjective, then $\tau_X$ and $\pi$ are factor maps. Moreover, 
  $\pi$ is almost $1$-$1$. Indeed, let $x\in X'$ and suppose that $(x',y')\in \tilde{X}$ is such that $\pi(x',y')=\pi(x, \hat{\pi_Z}(\phi(x)))$. This implies that $y'=\hat{\pi_Z}(\phi(x))$.   
Suppose that $x'\neq x$. Since $X$ is Hausdorff, there exist disjoint open neighborhoods $U\subseteq X$ and $V\subseteq X$ of $x$ and $x'$, respectively. The continuity of $\hat{\phi}$ and $\pi_Z$ imply that $\pi_Z^{-1}(\hat{\phi}^{-1}( U))$ is an open set of $Z$. Moreover, since $x\in U\cap X'$, we conclude that $\phi(x)\in \hat{\phi}^{-1}(U)\cap T$.
Thus, we have that $\hat{\pi_Z}(\phi(x))\in\pi_Z^{-1}(\phi^{-1}(U))$.
Observe that $V\times \pi_Z^{-1}(\hat{\phi}^{-1} (U))$ is an open neighborhood of $(x',\hat{\pi_Z}(\phi (x)))$.
Therefore, there exists $w\in X'$ such that $(w,\hat{\pi_Z}(\phi(w)))\in V\times \pi_Z^{-1}(\hat{\phi}^{-1} (U)) $.
Using that $\hat{\pi_Z}(\phi(w))\in \pi_Z^{-1}(\hat{\phi}^{-1}(U))$, we have $\phi(w)\in\hat{\phi}^{-1}(U)\cap T$ and consequently, $w\in U$. Hence, we obtain a contradiction with the fact that $U\cap V=\emptyset$. Thus, $x=x'$ and $\pi$ is an almost 1-1 factor.

Since $\pi_Z$ and $\pi$ are almost 1-1 factor maps, Lemma \ref{2-comp-almost} implies that $\pi_{\tilde{X}}=\pi_Z\circ \pi$ is an almost 1-1 factor map.
Finally, we will show that $\tilde{X}$ is minimal. Since the image by $\pi$ of every closed $G$-invariant set in $\tilde{X}$ is a closed $G$-invariant set in $Z$, the minimality of $Z$ implies that for every $\tilde{x}\in \tilde{X}$,  we have  $\pi(\overline{\{g\tilde{x}: g\in G\}})=Z$.
This implies that $\{\hat{\pi_Z}(\phi(x)):x\in X'\}\subseteq \pi(\overline{\{g\tilde{x}: g\in G\}})$. On the other hand, previously we have shown that every element in 
$\{\hat{\pi_Z}(\phi(x)):x\in X'\}$ has a unique pre-image by $\pi$, which implies that   
 
$\{(x,\hat{\pi_Z}(\phi(x))):x\in X'\}\subseteq \overline{\{g\tilde{x}: g\in G\}}$.  
Since $\{(x,\hat{\pi_Z}(\phi(x))):x\in X'\}$ is dense in $\tilde{X}$, we get that $\overline{\{g\tilde{x}: g\in G\}}=\tilde{X}$. Since $\tilde{X}$ is totally disconnected (it is a subspace of $X\times Z$) and infinite (it is an extension of $X$), by Lemma \ref{Cantor} we get that $(\tilde{X},G)$ is a minimal Cantor system.
 \begin{figure}
\begin{align}\label{figure2}
    \xymatrix{\tilde{X}\ar[r]^{\pi} \ar[d]_{\tau_X} \ar[dr]^-{\pi_{\tilde{X}}}& Z\ar[d]^-{\pi_Z}\\
X\ar[r]_{\phi}& Y }
\end{align}
\caption{Diagram corresponding to the different maps  of Lemma \ref{symbolic-factor}.}
\end{figure}
\end{proof}

\begin{proof}[{\bf Proof of Proposition \ref{main-tool}}]
Let $(Y,G)$, $(X,G)$ and $(Z,G)$ be the topological dynamical systems of Proposition \ref{main-tool}. 
Applying Lemma \ref{symbolic-factor} to these three systems, we get a minimal Cantor system $(\tilde{X},G)$ that is an extension of $(X,G)$. Since $(X,G)$ has no invariant probability measures, neither does $(\tilde{X},G)$. In addition, we get a factor map $\pi:\tilde{X}\to Z$, and an almost one-to-one factor map $\pi_{\tilde{X}}:\tilde{X}\to Y$, such that $\pi_{\tilde{X}}=\pi_Z\circ\pi.$

Now, applying Lemma \ref{sub-ext-1-1} to $(Y,G)$, $(\tilde{X},G)$, and  $(Z,G)$, we get there exists a minimal subshift $(\tilde{Z},G)$, with no invariant measures, which is an almost 1-1 factor of $(Y,G)$. 
\end{proof}

\begin{proof}[{\bf Proof of Theorem \ref{Main-Toeplitz-measures}}]
Let $G$ be a non-amenable residually finite group. Let $(Z,G)$ be a minimal equicontinuous Cantor system with a free orbit. Lemma \ref{equicontinuous-characterization} implies that $(Z,G)$ is conjugate to the odometer associated to a strictly decreasing sequence of finite index subgroups $(\Gamma_n)_{n\in\NN}$.   Since $G$ is non-amenable, \cite{GH97} implies there exists a minimal Cantor system $(\tilde{Y},G)$  with no invariant probability measures. Consider the system $(Y,G)$, where $Y=\tilde{Y}\times Z$ equipped with the product action. Let $W\subseteq Y$ be a minimal component of $(Y,G)$. Since $(W,G)$ is an extension of $(\tilde{Y},G)$, we have that $(W,G)$ is a minimal Cantor system with no invariant probability measures. Since $(W,G)$ is also an extension of $(Z,G)$,  applying Theorem \ref{theorem}, we get there exists a minimal Cantor system $(X,G)$ with no invariant probability measures, which is an almost one-to-one extension  of $(Z,G)$. By \ref{extension-Toeplitz}, $(Z,G)$ admits symbolic almost one-to-one extension. Thus, applying Proposition \ref{main-tool}, we get there exists a minimal subshift $(\tilde{Z},G)$, with no invariant probability measures, which is an almost one-to-one extension of $(Z,G)$. Finally, from Proposition \ref{1-1-extension} we deduce that $(\tilde{Z},G)$ is a Toeplitz $G$-subshift.
\end{proof}

\begin{proof}[{\bf Proof of Corollary \ref{corollary}}]
 The implications (1) to (2), (2) to (3), (3) to (4), and (4) to (5) are direct.    The implication (4) to (1) is  consequence of Theorem \ref{Main-Toeplitz-measures} and Proposition \ref{group-characterization}.
\end{proof}

\end{document}